\documentclass[12pt]{amsart}
\usepackage{txfonts}
\usepackage{amscd,amssymb,mathabx,subfigure}
\usepackage{breakurl}
\usepackage[all]{xy}
\usepackage{graphicx,calrsfs}
\usepackage[colorlinks,plainpages,backref,urlcolor=blue]{hyperref}
\usepackage{tikz}
\usetikzlibrary{calc}
\usepackage{mdwlist}
\usepackage{enumitem}

\newtheorem{theorem}{Theorem}[section]
\newtheorem{corollary}[theorem]{Corollary}
\newtheorem{lemma}[theorem]{Lemma}
\newtheorem{prop}[theorem]{Proposition}

\theoremstyle{definition}
\newtheorem{definition}[theorem]{Definition}
\newtheorem{example}[theorem]{Example}
\newtheorem{remark}[theorem]{Remark}

\newcommand{\N}{\mathbb{N}}
\newcommand{\Z}{\mathbb{Z}}
\newcommand{\Q}{\mathbb{Q}}
\newcommand{\R}{\mathbb{R}}
\newcommand{\C}{\mathbb{C}}

\renewcommand{\L}{\mathbb{L}}

\newcommand{\CP}{\mathbb{CP}}
\renewcommand{\k}{\Bbbk}

\newcommand{\HH}{\mathbf{H}}

\DeclareMathAlphabet{\pazocal}{OMS}{zplm}{m}{n}

\newcommand{\B}{{\pazocal{B}}}

\newcommand{\NN}{{\pazocal{N}}}

\newcommand{\RR}{{\mathcal R}}
\newcommand{\VV}{{\mathcal V}}

\newcommand{\F}{{\mathcal{F}}}

\newcommand{\cE}{{\mathcal{E}}}
\newcommand{\M}{{\mathcal{M}}}
\newcommand{\X}{{\mathcal{X}}}
\newcommand{\G}{{\mathcal{G}}}
\newcommand{\LL}{{\mathcal{L}}}
\newcommand{\CC}{{\mathcal{C}}}
\renewcommand{\AA}{{\mathcal{A}}}

\newcommand{\g}{{\mathfrak{g}}}
\newcommand{\h}{{\mathfrak{h}}}
\newcommand{\gl}{{\mathfrak{gl}}}
\renewcommand{\sl}{{\mathfrak{sl}}}
\newcommand{\m}{{\mathfrak{m}}}

\newcommand{\sol}{{\mathfrak{sol}}}

\newcommand{\ii}{\mathrm{i}}
\newcommand{\art}{\text{{\sffamily A}}}

\DeclareMathOperator{\gr}{gr}
\DeclareMathOperator{\im}{im}

\DeclareMathOperator{\id}{id}
\DeclareMathOperator{\ab}{{ab}}

\DeclareMathOperator{\GL}{GL}
\DeclareMathOperator{\SL}{SL}
\DeclareMathOperator{\PSL}{PSL}

\DeclareMathOperator{\Hom}{{Hom}}

\DeclareMathOperator{\ev}{ev}
\DeclareMathOperator{\Ev}{Ev}

\DeclareMathOperator{\ad}{ad}

\DeclareMathOperator{\ideal}{ideal}

\DeclareMathOperator{\orb}{orb}
\DeclareMathOperator{\abf}{abf}
\DeclareMathOperator{\MG}{MG}
\DeclareMathOperator{\Def}{Def}
\DeclareMathOperator{\mon}{mon}
\DeclareMathOperator{\DR}{dR}
\DeclareMathOperator{\Alb}{Alb}

\newcommand{\surj}{\twoheadrightarrow}
\newcommand{\inj}{\hookrightarrow}

\newcommand\isom{\xrightarrow{
 \,\smash{\raisebox{-0.4ex}{\ensuremath{\scriptstyle\simeq}}}\,}}
\newcommand{\abs}[1]{\left| #1 \right|}

\newcommand{\DS}{\displaystyle}

\def\dot{\mathchar"013A}  
\newcommand{\hdot}{{\raise2pt\hbox to0.35em{\Huge $\dot$}}} 
\newcommand{\sdot}{{\hbox to0.35em{\Huge $\dot$}}} 
\newcommand{\bwedge}{\mbox{\small $\bigwedge$}}
\newcommand{\bwedgedot}{\bwedge\!^{\hdot}}

\let\emptyset\varnothing
\newcommand{\Ka}{K\"{a}hler }
\newcommand{\bmp}{\overline{M}\,\!'}

\newenvironment{romenum}
{ 

\begin{enumerate}}{\end{enumerate}}

\newenvironment{alphenum}
{

\begin{enumerate}}{\end{enumerate}}

\newcommand{\CDGA}{\textup{\texttt{CDGA}}}
\newcommand{\DGL}{\textup{\texttt{DGL}}}
\newcommand{\ACDGA}{\textup{\texttt{ACDGA}}}
\newcommand{\FDGA}{\textup{\texttt{FDGA}}}
\newcommand{\MHD}{\textup{\texttt{MHD}}}
\newcommand{\Lie}{\textup{\texttt{Lie}}}
\newcommand{\Art}{\textup{\texttt{Art}}}
\newcommand{\Tp}{\textup{\texttt{Top}}}
\newcommand{\Set}{\textup{\texttt{Set}}}
\newcommand{\Cat}{\textup{\texttt{C}}}

\newcommand{\ga}{\large{\textsc{cga}}}
\newcommand{\cdga}{\large{\textsc{cdga }}}
\newcommand{\dga}{\large{\textsc{cdga}}}
\newcommand{\acdga}{\large{\textsc{acdga }}}
\newcommand{\adga}{\large{\textsc{acdga}}}
\newcommand{\dgl}{\large{\textsc{dgl}}}

\newcommand{\fdga}{\large{\textsc{fdga}}}

\definecolor{dkgreen}{RGB}{0,100,0}
\definecolor{dkbrown}{RGB}{139,69,19}

\numberwithin{equation}{section}

\begin{document}

\title[Topological and infinitesimal embedded jump loci]{%
Naturality properties and comparison results for topological 
and infinitesimal embedded jump loci}

\author[Stefan Papadima]{Stefan Papadima$^1$$\dagger$}
\address{Simion Stoilow Institute of Mathematics, 
P.O. Box 1-764,
RO-014700 Bucharest, Romania}
\email{\href{mailto:Stefan.Papadima@imar.ro}{Stefan.Papadima@imar.ro}}
\thanks{$^1$Work partially supported by the Romanian Ministry 
of Research and Innovation, CNCS--UEFISCDI, grant
PN-III-P4-ID-PCE-2016-0030, within PNCDI III}
\thanks{$\dagger$Deceased January 10, 2018}

\author[Alexander~I.~Suciu]{Alexander~I.~Suciu$^2$}
\address{Department of Mathematics,
Northeastern University,
Boston, MA 02115, USA}
\email{\href{mailto:a.suciu@northeastern.edu}{a.suciu@northeastern.edu}}
\urladdr{\href{http://web.northeastern.edu/suciu/}%
{web.northeastern.edu/suciu/}}
\thanks{$^2$Partially supported by the Simons Foundation 
collaboration grant for mathematicians 354156}



\begin{abstract}
We use augmented commutative differential graded algebra ($\adga$) models 
to study $G$-representation varieties of fundamental groups $\pi=\pi_1(M)$ 
and their embedded cohomology jump loci, around the trivial 
representation $1$.  When the space $M$ admits a finite family 
of maps, uniformly modeled by $\adga$ morphisms, and 
certain finiteness and connectivity assumptions are satisfied,  
the germs at $1$ of $\Hom (\pi,G)$ and of the embedded 
jump loci can be described in terms of their infinitesimal counterparts,
naturally with respect to the given families. 
This approach leads to fairly explicit answers 
when $M$ is either a compact K\"{a}hler manifold, the complement 
of a central complex hyperplane arrangement, or the total space of a principal 
bundle with formal base space, provided the Lie algebra of the linear 
algebraic group $G$ is a non-abelian
subalgebra of $\sl_2(\C)$.
\end{abstract}

\maketitle
\setcounter{tocdepth}{1}
\tableofcontents

\section{Introduction and statement of results}
\label{sect:intro}

\subsection{Representation varieties and jump loci}
\label{subsec:rep jumps}
Sheaf cohomology is ubiquitous in geometry and topology. The parameter 
space for rank $n$ locally constant sheaves on a path-connected, pointed 
CW-complex $X$ with finitely many $1$-cells may be identified with the 
$\GL_n$-representation variety of the fundamental group of $X$.  
Twisted cohomology on $X$ is encoded by the filtrations of these varieties 
by the (embedded) jump loci.  While $\GL_1$-representation varieties 
are finite unions of affine tori, the picture 
changes dramatically in higher rank.  For instance, the universality theorem 
of Kapovich and Millson \cite{KM} states that $\PSL_2$-representation varieties 
may have arbitrarily bad singularities, away from the origin $1$ (the trivial representation). 
This is the reason why we focus here on analytic germs at the origin of the embedded
cohomology jump loci. The general case is analyzed by Budur and Wang in \cite{BW}, 
but it seems that explicit computations away from $1$ are intractable in full generality. 

By the main result from \cite{DP-ccm}, Theorem B, the germs at $1$ of the 
embedded jump loci of $X$ are isomorphic to the germs at $0$ of the 
infinitesimal embedded jump loci of a commutative, differential graded algebra $A$, 
provided this $\dga$ models $X$, and certain mild finiteness assumptions  
are satisfied. Furthermore, in the abelian case, this identification is natural.  
One of our main goals here is to extend this natural comparison to the 
non-abelian setting, by studying the behavior of jump loci under suitable 
continuous maps between spaces and $\cdga$ maps between their models. 

By construction, both representation varieties and their infinitesimal analogues 
are (bi)functorial.  The naturality properties of both types of jump loci are 
summarized in Corollaries \ref{cor:jumpnat-top} and  \ref{cor:jumpnat-inf}.  
As we point out in Example \ref{ex:nonat}, naturality at this level requires 
connectivity assumptions for maps, defined in \S\ref{subsec:qconn}.
What greatly simplifies things in the abelian case is the existence of a global, 
exponential map which relates representation varieties to their infinitesimal 
counterparts.  By way of contrast, it follows from Example \ref{ex:nolift} that 
no such map exists in the $\PSL_2$ case, even for compact Riemann surfaces. 

To avoid this major difficulty, we construct local analytic isomorphisms between 
the two types of jump loci by means of Artin approximation.  That is, we replace 
the respective local analytic rings by their completions (or by functors of Artin rings), 
and deduce local analytic naturality from naturality at the level of completions. 
This we do in Proposition \ref{prop:simartin}, which is a general result about 
simultaneous Artin approximation.   We need this `simultaneous' framework in 
view of the applications to be derived later on, which involve families of maps 
between spaces.  

The condition that makes our approach to a natural comparison between 
embedded jump loci work is based on the $q$-equivalence relation for 
morphisms between augmented commutative differential graded algebras ($\adga$s), 
denoted by $\simeq_q$, and detailed in Definition \ref{def:qeqmaps}. 
The primary examples we have in mind are the $\adga$ morphisms 
$\Omega(f)$ between Sullivan--de Rham models induced by pointed, 
continuous maps between topological spaces. 

\subsection{Natural comparison with respect to finite families of maps}
\label{subsec:nat family}
In more detail, the embedded jump loci we consider in this paper are as follows.
First let $X$ be a pointed, path-connected space with fundamental group $\pi$, and
let $\iota\colon G \to \GL(V)$ be a representation. For each $i,r\ge 0$, the embedded
jump locus of $X$ with respect to $\iota$ is the pair 
\begin{equation}
\label{eq:embjl-intro}
\left (\Hom(\pi,G),\VV^i_r(X,\iota) \right), 
\end{equation}
where $\VV^i_r(X,\iota)$ is the set of homomorphisms $\rho$ 
for which the $i$-th cohomology group of $X$ with coefficients in the 
local system $V_{\iota\circ \rho}$ has dimension at least $r$.

Next, let $A$ be a $\dga$, and let $\theta\colon \g\to \gl(V)$ be a
Lie algebra representation. The infinitesimal analog of the representation variety is
the set  $\F(A,\g)\subseteq A^1\otimes \g$ of $\g$-valued flat connections on $A$. 
For each $i,r\ge 0$, the 
infinitesimal embedded jump locus of $A$ with respect to $\theta$ is the pair 
\begin{equation}
\label{eq:infjl-intro}
\left( \F(A,\g) , \RR^i_r(A, \theta) \right), 
\end{equation}
where $\RR^i_r(A,\theta)$ is the set of flat connections $\omega$ 
for which the $i$-th cohomology group of the cochain complex 
$(A\otimes V, d_{\omega})$ defined in \eqref{eq:aomoto} has dimension at least $r$.

Assume now that $\iota$ is a rational representation of linear algebraic 
groups, and both $\g$ and $V$ are finite-dimensional. Under mild 
$q$-finiteness conditions on $X$ and $A$ (explained in \S\ref{subsec:models}),
 both $\Hom(\pi,G)$ and $\F(A,\g)$ are affine varieties, and their jump loci 
 are closed subvarieties, for all $i\le q$ and $r\ge0$.

We may now state our first main result.  Let $\{f\colon X\to X_f\}_{f\in E}$ be a 
finite family of continuous maps between pointed, path-connected spaces.  
For each $f\in E$, we denote by $f_{\sharp} \colon \pi\to \pi_f$ the induced 
homomorphism on fundamental groups. Let also $\{\Phi_f\colon A_f\to A\}_{f\in E}$ 
be a family of $\adga$ morphisms. Consider a rational representation 
of linear algebraic groups, $\iota\colon G\to \GL(V)$, over $\k=\R$ or $\C$, 
with tangential representation $\theta\colon \g\to \gl(V)$.  
For an affine $\k$-variety $\X$, we denote by $\X_{(x)}$ the 
$\k$-analytic germ of $\X$ at a point $x\in \X$. 

\begin{theorem}
\label{thm:main-intro}
Fix an integer $q\ge 1$,
and suppose the following conditions hold:
\begin{enumerate}
\item All the above spaces and $\dga$s are $q$-finite.
\item Both $f$ and $\Phi_f$ are $(q-1)$-connected maps, for all $f\in E$. 
\item $\Omega(f)\simeq_q \Phi_f$ in $\ACDGA$, uniformly with respect to $f\in E$.
\end{enumerate}
Under these assumptions, we may find  local analytic isomorphisms 
$a\colon \F(A,\g)_{(0)} \isom \Hom(\pi,G)_{(1)}$ 
and $a_f\colon \F(A_f,\g)_{(0)} \isom \Hom(\pi_f,G)_{(1)}$ for all $f\in E$ 
with the property that the following diagram commutes, for all $f\in E$:
\[
\xymatrix{
\F(A,\g)_{(0)} \ar[r]^(.45){a}& \Hom(\pi,G)_{(1)}\\
\F(A_f,\g)_{(0)} \ar[u]^{\Phi_f\otimes \id} \ar[r]^(.45){a_f}
&  \Hom(\pi_f,G)_{(1)}\ar[u]_{f_{\sharp}^{!}}.
}
\]
Moreover, this construction induces the following commuting diagram of 
(local, reduced) embedded jump loci, for all $f\in E$, $i\le q$, and $r\ge 0$:
\[
\xymatrix{
(\F(A,\g), \RR^i_r(A,\theta))_{(0)} \ar[r]^(.48){a}& (\Hom(\pi,G), \VV^i_r(X,\iota))_{(1)}\,\,
\\
(\F(A_f,\g), \RR^i_r(A_f,\theta))_{(0)}\ar[u]^{\Phi_f\otimes \id} \ar[r]^(.48){a_f}
&  (\Hom(\pi_f,G), \VV^i_r(X_f,\iota))_{(1)}\, ,   \ar[u]_{f_{\sharp}^{!}}
}
\]
where both horizontal arrows are isomorphisms of analytic pairs. 
\end{theorem}

The meaning of the $q$-equivalence relation between continuous pointed maps 
and $\adga$ maps, uniformly with respect to finite families, is explained in Definition 
\ref{def:unifq}. For a one-element family $\{f\}$, this condition simply means that 
$\Omega(f)\simeq_q \Phi_f$ in $\ACDGA$.  For a certain type of two-element 
family, the uniformity condition is verified in the next theorem.

Let $f\colon X\to Y$ be a continuous, pointed map between path-connected spaces.  
Let $\pi$ be the fundamental group of $X$, let $\abf\colon \pi\surj \pi_{\abf}$ 
be the projection onto its maximal torsion-free abelian quotient, and 
let $f_0\colon X\to K(\pi_{\abf},1)$ be a classifying map for this   
projection. Set  $A^{\hdot}_0= (\bwedgedot H^1(X), d=0)$.  

\begin{theorem}
\label{thm:2unif-intro}
Suppose that 
$X$ and $Y$ are $q$-finite, for some $q\ge 1$, and $\Omega(f)\simeq_q \Phi$ in 
$\ACDGA$, where $\Phi\colon A_Y\to A_X$ is a morphism between $q$-finite 
$\adga$s.   There is then an $\adga$ map 
$\Phi_0\colon A_0 \to A_X$ inducing an isomorphism on $H^1$, and such that 
$\Omega(f_0)\simeq_q \Phi_0$ in $\ACDGA$, uniformly with respect to the 
families $\{f,f_0\}$ and $\{\Phi,\Phi_0\}$. Moreover, if $f$ and $\Phi$ are $0$-connected 
maps, then all hypotheses from Theorem \ref{thm:main-intro} are satisfied for $q=1$. 
\end{theorem}

\subsection{A general framework for applications}
\label{subsec:apps}
Let $\pi$ be the fundamental group of a $1$-finite manifold $M$.  
We aim at finding structural results for (non-abelian) embedded jump loci of $M$ near the 
origin, in low degrees. To start with, we want  to extract from the geometry of $M$ 
a finite family of group epimorphisms, $\{f_{\sharp}\colon \pi\surj \pi_f\}_{f\in \cE(M)}$, induced 
on fundamental groups by maps $f\colon M\to M_f$ onto manifolds of smaller dimension. 
Next, we set $E(M)=\cE(M)\cup \{f_0\}$, where $f_0$ is a classifying map for 
the projection $\abf \colon \pi\surj \pi_{\abf}$.  

Let $\iota\colon G\to \GL(V)$ be a rational representation of $\C$-linear algebraic groups.
For a group homomorphism $h\colon \pi\to \pi'$, we let 
$h^{!}\colon \Hom(\pi',G)\to \Hom(\pi,G)$ denote the induced morphism 
between the corresponding representation varieties. 
As explained in Remark \ref{rem:tri}, the abelian part of $\Hom(\pi,G)$ near the 
trivial representation coincides with the germ $\abf^{!} (\Hom(\pi_{\abf} , G))_{(1)}$.
By naturality of jump loci, the natural inclusion,
\begin{equation}
\label{eq:repincl-intro}
\Hom(\pi,G) \supseteq \bigcup_{f\in E(M)} f_{\sharp}^{!} \Hom (\pi_f,G)\, ,
\end{equation}
induces inclusions 
\begin{equation}
\label{eq:vincl-intro}
\VV^i_r(M,\iota) \supseteq \bigcup_{f\in E(M)} f^{!} \VV^i_r (M_f,\iota)
\end{equation}
for all $i\le 1$ and $r\ge 1$. Finally, we 
ask whether the inclusions \eqref{eq:repincl-intro} and \eqref{eq:vincl-intro} 
become equalities near $1$.

We focus on the rank $2$ case, when the Lie algebra $\g$ of $G$ is a non-abelian 
subalgebra of $\sl_2(\C)$.  Our techniques allow us to treat simultaneously two 
interesting classes of examples: (1) quasi-compact K\"{a}hler manifolds 
(in particular, quasi-projective manifolds), and (2) closed, 
smooth manifolds endowed with a free action of a  compact, connected, real 
Lie group $K$. In the first case, the family $\cE(M)$ consists of equivalence classes 
of `admissible' maps (in the sense of Arapura \cite{Ar}) from $M$ to smooth complex 
curves of negative Euler characteristic.  In the second case, $\cE(M)$ has only one 
element, namely the bundle projection $M\to M/K$.  

When the group $G$ is $\SL_2(\C)$ or $\PSL_2(\C)$ and $M$ is a quasi-projective 
manifold, equality in \eqref{eq:repincl-intro} is related to deep results of Corlette--Simpson 
\cite{CS} and  Loray--Pereira--Touzet \cite{LPT}, which give a rather intricate  
classification for the $G$-representations of $\pi_1(M)$, 
also valid away from $1$.  When $M$ is a quasi-compact K\"{a}hler manifold, 
$\iota=\id_{\C^{\times}}$, and $i=r=1$,  equality in \eqref{eq:vincl-intro} near $1$  
is equivalent to the subtle description of $\VV^1_1(M, \iota)$ from \cite{Ar}, again 
also valid away from $1$.  Thus, our results below may be viewed as a 
more precise version of the aforementioned work, in a broader context, 
albeit only near the origin. 

\begin{theorem}
\label{thm:intro2}
Let $G$ be a $\C$-linear algebraic group with non-abelian Lie algebra 
$\g\subseteq \sl_2(\C)$, and let $\iota\colon G\to \GL(V)$ be a rational 
representation.  For $i=r=1$ and for $i=0$, $r\ge 1$, both 
\eqref{eq:repincl-intro} and \eqref{eq:vincl-intro} become 
equalities near the origin $1$, provided $\pi=\pi_1(M)$ and either
\begin{enumerate}
\item \label{di1} 
$M$ is a compact, connected K\"{a}hler manifold;
\item  \label{di2} 
$M$ is the complement of a (central) complex hyperplane arrangement;
\item  \label{di3}  
$M$ is a closed, connected, differentiable manifold supporting a free action by a 
compact, connected real Lie group $K$, and the orbit space $M/K$ is formal
in the sense of Sullivan \cite{Su}.
\end{enumerate}
Here $E(M)=\cE(M)\cup \{f_0\}$, where $f_0$ realizes 
$\abf \colon \pi\surj \pi_{\abf}$ 
and the set $\cE(M)$ consists of all admissible maps in the
first two cases, and the projection $M \to M/K$ in the third.
\end{theorem}

The common strategy of proof is to choose appropriate uniform $\adga$ models 
for the family $E(M)$ and apply Theorem \ref{thm:main-intro} to replace topological 
by infinitesimal equalities.  In turn, the latter equalities are verified using results 
from \cite{MPPS} for parts \eqref{di1}--\eqref{di2} and from \cite{PS-15} for part \eqref{di3}.

Compact K\"{a}hler manifolds and complements of complex hyperplane arrangements 
provide highly non-trivial examples of uniform formality (over $\k=\R$ or $\C$) 
with respect to finite families of maps. In Proposition \ref{prop:unifk}, we reinterpret 
the main result from \cite{DGMS} in the following form: $\Omega_{\k}(f)\simeq H^{\hdot}(f,\k)$ 
in $\ACDGA$, uniformly with respect to an arbitrary finite family of holomorphic maps 
between compact K\"{a}hler manifolds.  Similarly, we recast in Proposition \ref{prop:unifarr} 
the main result of \cite{FY}, as follows:   $\Omega_{\k}(f)\simeq H^{\hdot}(f,\k)$ 
in $\ACDGA$, uniformly with respect to the family $E(M_{\AA})$, for any central 
complex hyperplane arrangement $\AA$ in $\C^3$ with complement $M_{\AA}$.  In this way, 
we are able to apply Theorem \ref{thm:main-intro} in order to 
prove parts \eqref{di1}--\eqref{di2} of Theorem \ref{thm:intro2}.

For part \eqref{di3}, let $N=M/K$ be the orbit space of the free $K$-action on $M$, 
and let $f\colon M\to N$ be the projection map of the resulting principal $K$-bundle. 
Assuming that $N$ has a finite model $A_N$ over a field $\k$ of characteristic $0$, 
we construct in Proposition \ref{prop:hmodel} a finite model $A_M$ for $M$ and an 
$\adga$ map $\Phi_f\colon A_N\to A_M$ such that $\Omega_{\k}(f)\simeq \Phi_f$ 
in $\ACDGA$.  In the case when $N$ is formal, we may take $A_N=(H^{\hdot}(N,\k),d=0)$.  
Applying now Theorems \ref{thm:main-intro} and \ref{thm:2unif-intro} 
completes the proof of Theorem \ref{thm:intro2}\eqref{di3}.

Further applications of the techniques that go into proving the above results 
can be found in our recent preprint \cite{PS-17}.  In particular, in the context of 
Theorem \ref{thm:intro2}, parts \eqref{di1}--\eqref{di2}, but for an arbitrary complex 
linear algebraic group $G$, it is shown in \cite[Theorem 1.1(2)]{PS-17} that the germs 
$f_{\sharp}^{!} \Hom (\pi_f,G)_{(1)}$ and  $g_{\sharp}^{!} \Hom (\pi_g,G)_{(1)}$ 
from decomposition \eqref{eq:repincl-intro}  intersect only at the origin, provided 
the maps $f,g\in \cE(M)$ are distinct. This transversality property is a substantial 
non-abelian extension of the corresponding rank $1$ result, proved in \cite{DPS-duke}
in the case when $\iota$ is the standard isomorphism $\C^{\times} \isom \GL_1(\C)$.

\subsection{Formal maps and regular maps}
\label{subsec:fregmaps}

The uniformity property for one-element families of maps may be 
verified in two further classes of examples: formal maps between 
formal spaces, and regular maps between quasi-projective manifolds. 

By definition, a continuous map $f\colon X \to Y$ is formal over $\k$ if it is 
modeled in $\CDGA$ by the morphism 
$H^{\hdot}(f,\k)\colon H^{\hdot}(Y,\k)\to H^{\hdot}(X,\k)$, cf.~\cite{Su, VP}.  
In Proposition \ref{prop:mainf}, we prove the following:  
If $f$ is formal over $\k$ and $H^1(f,\k)$ is injective, 
then $\Omega_{\k}(f) \simeq H^{\hdot}(f,\k)$ in $\k$-$\ACDGA$.  
Applying Theorem \ref{thm:main-intro} yields relevant information 
(summarized in Proposition \ref{prop:fmap})
on the map induced by $f$ between the corresponding 
embedded jump loci. 

To state the quasi-projective analogue of the formality property, we need to 
recall from \cite{Mo, CG} some relevant facts.  Every quasi-projective manifold $M$ 
is of the form $\overline{M}\setminus D$, where $\overline{M}$ is a smooth, projective variety 
and $D$ is a normal-crossing divisor in $\overline{M}$. A regular map between two such  
manifolds, $f\colon M\to M'$, is induced by a regular map $\bar{f}\colon \overline{M} \to 
\bmp$ with the property that $\bar{f}^{-1}(D')\subseteq D$.  The manifold $M$ 
admits as a finite $\dga$ model over $\C$ Morgan's Gysin model $\MG(\overline{M}, D)$.  
Furthermore, the regular map $f$ is modeled in $\C$-$\CDGA$ by a certain map 
$\Phi (\bar{f})\colon \MG(\bmp, D') \to \MG(\overline{M}, D)$.  

In Proposition \ref{lem:mainqp}, we use relative Sullivan models for 
$\dga$ maps to  improve on these known facts, as follows.  
Let $f\colon M\to M'$ be a regular map between quasi-projective manifolds, 
and let $\bar{f}\colon (\overline{M},D) \to (\bmp,D')$ be an extension as above.  
If $H^1(f)$ is injective, then $\Omega_{\C}(f) \simeq \Phi (\bar{f})$ in $\C$-$\ACDGA$.
We indicate in Remark \ref{rem:deligne} some possible applications of this result.

\subsection{Conventions}
\label{subsec:conv}
All spaces are assumed to be path-connected.  The default 
coefficient ring is a field $\k$ of characteristic $0$.  (When speaking 
about analytic germs and analytic algebras, $\k$ will be either $\R$ or $\C$.) 
Graded $\k$-vector spaces are non-negatively graded.  

\section{Artin approximation}
\label{sect:artin}

Our approach to naturality properties of cohomology jump loci 
is based on (simultaneous) Artin approximation, using the 
book Tougeron~\cite{T} as a basic reference.  
We start by proving a general result of this type. 

Given a local ring $(R,\m)$, we denote by $\gr_{\hdot}(R)$ the 
associated graded ring with respect to the $\m$-adic filtration. 
The $\m$-adic completion of $R$ will be denoted by $\widehat{R}$.  
If $I \subset R$ is an ideal, $\widehat{I}$ will stand for the extended 
ideal $\widehat{R}\cdot I$ of $\widehat{R}$.  Morphisms between 
local rings are assumed to be local. 

We will use M.~Artin's theorem on approximating formal power series 
solutions of analytic equations by convergent power series (see 
\cite[III.4]{T}) in the following form.

\begin{theorem}
\label{thm:martin}
Let $R$ and $\overline{R}$ be two analytic algebras, and 
let $\{I_k\}_{k\in F}$ and $\{\overline{I}_k\}_{k\in F}$ be two 
finite families of proper ideals in these algebras. Suppose 
$\alpha\colon \widehat{R} \to \widehat{\overline{R}}$ is 
a morphism sending $\widehat{I}_k$ to $\widehat{\overline{I}}_k$ 
for all $k$. There is then a morphism $a\colon R \to \overline{R}$ 
such that $a(I_k)\subset \overline{I}_k$ for all $k$ and 
$\gr_1(a)=\gr_1(\alpha)$. 
\end{theorem}

The next lemma will also be useful in the sequel. 

\begin{lemma}
\label{lem:gr}
Let $R$ and $\overline{R}$ be two analytic algebras, and 
let $I\subset R$ and $\overline{I} \subset \overline{R}$ be two 
proper ideals. Suppose $a\colon R \to \overline{R}$ is a 
morphism that sends $I$ to $\overline{I}$, and 
$\alpha\colon \widehat{R} \to \widehat{\overline{R}}$ is 
an isomorphism 
such that $\alpha(\widehat{I})= \widehat{\overline{I}}$ and 
$\gr_1(a)=\gr_1(\alpha)$. Then $a$ is an isomorphism 
and $a(I)=\overline{I}$.
\end{lemma}

\begin{proof}
It follows from \cite[III.5]{T} that a morphism of analytic algebras 
with isomorphic completions must be an isomorphism, provided 
the given morphism induces a surjection on $\gr_1$. Consequently, 
both $a\colon R \to \overline{R}$ and the induced morphism, 
$a'\colon R/I \to \overline{R}/\overline{I}$, 
are isomorphisms, and the claim follows. 
\end{proof}

We are now ready to describe the setup for our approximation 
result.  Let $\{\phi_i\colon R\surj R_i\}_{i\in E}$ and 
$\{\bar\phi_i\colon \overline{R}\surj \overline{R}_i\}_{i\in E}$ be two 
families of epimorphisms between analytic algebras, indexed by the same finite set $E$. 
Furthermore, let $\{I^j \subseteq R\}_{j\in F}$, 
$\{\overline{I}^j \subseteq \overline{R}\}_{j\in F}$, 
$\{I^j_i \subseteq R_i\}_{j\in F}$, and 
$\{\overline{I}^j_i \subseteq \overline{R}_i\}_{j\in F}$ be 
families of ideals in the respective analytic algebras, 
indexed by the same finite set $F$.  
Finally, let $\alpha \colon \widehat{R}\isom \widehat{\overline{R}}$ and 
$\{\alpha_i \colon \widehat{R}_i\isom \widehat{\overline{R}}_i\}_{i\in E}$
be isomorphisms between the respective completions.

\begin{prop}
\label{prop:simartin}
In the above setup, assume the following 
conditions hold, for all  $i\in E$ and $j\in F$ (as the case may be):
\begin{enumerate}
\item \label{p1}
$I^j\ne R \Leftrightarrow \overline{I}^j\ne \overline{R}, \, 
I^j_i\ne R_i \Leftrightarrow \overline{I}^j_i\ne \overline{R}_i, \,  
I^j = R \Rightarrow I^j_i = R_i, \, 
\overline{I}^j = \overline{R} \Rightarrow \overline{I}^j_i = \overline{R}_i$ ;
\item \label{p2}
$\alpha(\widehat{I}^j) = \widehat{\overline{I}}^j$ ; 
\item \label{p3}
$\alpha_i(\widehat{I}^j_i) = \widehat{\overline{I}}^j_i$ ;
\item  \label{p4}
$\hat{\bar\phi}_i \circ \alpha = 
\alpha_i\circ \hat\phi_i$ . 
\end{enumerate}
There exist then isomorphisms 
$a \colon R\isom \overline{R}$ and 
$\{a_i \colon R_i\isom \overline{R}_i\}_{i\in E}$ such that 
\begin{romenum}
\item \label{i1}
$a(I^j) = \overline{I}^j$ \/
for all $j\in F$;
\item  \label{i2}
$a_i(I^j_i) = \overline{I}^j_i$ \/
for all $i\in E$ and $j\in F$;
\item  \label{i3}
$\bar\phi_i \circ a = a_i\circ \phi_i$ \/ for all $i\in E$.
\end{romenum}
\end{prop}

\begin{proof}
Without loss of generality, we may assume all ideals in sight are 
proper, replacing if need be the set $F$ by subsets $F_0\subseteq F$ 
and $F_i\subseteq F_0$ for $i\in E$.  
For each $i\in E$, put $K_i=\ker (\phi_i)$ and 
$\overline{K}_i=\ker (\bar\phi_i)$, and pick proper ideals, 
$\{J^j_i\subset R\}_{j\in F_i}$ and 
$\{\overline{J}^j_i\subset \overline{R}\}_{j\in F_i}$, such that 
$\phi_i(J^j_i)=I^j_i$ and $\bar\phi_i(\overline{J}^j_i)=\overline{I}^j_i$. 
We claim that it is enough to find a morphism $a\colon R\to \overline{R}$ 
such that 
\begin{alphenum}
\item \label{s1}
$\gr_1(a)=\gr_1(\alpha)$;
\item \label{s2}
$a(K_i)\subseteq \overline{K}_i$, for all $i\in E$;
\item \label{s3}
$a(I^j)\subseteq \overline{I}^j$, for all $j\in F_0$;
\item \label{s4}
$a(J^j_i)\subseteq \overline{J}^j_i +  \overline{K}_i$, for all $(i,j)\in E\times F_i$.
\end{alphenum}

Indeed, by \eqref{s2}, the morphism $a\colon R\to \overline{R}$ induces 
morphisms $a_i\colon R_i\to \overline{R}_i$ for all $i\in E$. 
In view of \eqref{s1}, we may apply Lemma \ref{lem:gr} and  
deduce that the map $a$ is an isomorphism. By construction, 
property \eqref{i3} is satisfied. By assumption, equality \eqref{p4} holds,   
and so $\alpha(\widehat{K}_i)=\widehat{\overline{K}}_i$. 
Again by Lemma \ref{lem:gr}, the 
maps $a_i$ must be isomorphisms, for all $i\in E$.
In view of \eqref{s3}, we may also apply Lemma \ref{lem:gr} 
to each of the ideals $I^j\subset R$ and $\overline{I}^j\subset \overline{R}$ 
for $j\in F_0$, and deduce that $a(I^j)=\overline{I}^j$, 
thereby verifying property \eqref{i1}. 

Finally, to verify property \eqref{i2}, we apply Lemma \ref{lem:gr} to the 
ideals $I^j_i=(J^j_i + K_i)/K_i \subset R/K_i=R_i$ and 
$\overline{I}^j_i=(\overline{J}^j_i +  \overline{K}_i)/ \overline{K}_i  
\subset \overline{R}/\overline{K}_i=\overline{R}_i $ 
for $(i,j)\in E\times F_i$.  We know from \eqref{s2} and \eqref{s4} 
that the morphisms $a_i\colon R_i\to \overline{R}_i$ preserve 
these ideals. The fact that the isomorphisms $\alpha_i$ identify the completions
of these ideals follows from their construction, together with assumption \eqref{p3}.
Moreover, since $a_i$ is induced by $a$ and $\alpha_i$ is induced by $\alpha$, 
property \eqref{s1} implies that $\gr_1(a_i)=\gr_1(\alpha_i)$.  Thus, 
Lemma \ref{lem:gr} applies once again to show that 
$a_i(I^j_i)=\overline{I}^j_i$.  
This completes the verification of our claim. 

Assumptions \eqref{p2}--\eqref{p4} insure that the map 
$\alpha \colon \widehat{R}\isom \widehat{\overline{R}}$ is 
a formal series solution of the analytic system \eqref{s2}--\eqref{s4}. 
Applying now Theorem \ref{thm:martin} completes the proof.
\end{proof}

\section{Algebraic models of spaces and maps}
\label{sect:cgga}

The rational homotopy theory of Quillen \cite{Qu}, as reinterpreted by 
Sullivan in \cite{Su}, provides a very useful mechanism for 
studying topological properties of spaces and continuous 
maps by considering commutative differential graded algebra 
(for short, $\dga$) models for them.  In this section, we review 
the basics of this theory, and draw some consequences in the 
formal case. 

\subsection{$q$-connectivity and $q$-equivalences}
\label{subsec:qconn}
We start with some basic terminology, related to 
connectivity properties of spaces and $\dga$s. Fix $0\le q\le \infty$, 
usually to be omitted from notation when $q=\infty$. Let 
$\psi\colon C^{\hdot} \to C'^{\hdot}$ be a morphism of 
graded vector spaces.  We say $\psi$ is {\em $q$-connected}\/ 
if it is an isomorphism in degrees up to $q$ and a monomorphism 
in degree $q+1$.  

When such a map $\psi$ is a $q$-connected morphism of 
cochain complexes, it is straightforward to check that the 
induced morphism in cohomology, 
$H^{\hdot}(\psi)\colon H^{\hdot} (C) \to H^{\hdot} (C')$, 
is again $q$-connected.  Cochain maps inducing a $q$-connected 
map in cohomology will be called {\em $q$-equivalences}. For $q=\infty$,
these maps are also called quasi-isomorphisms in the literature.

We say that a commutative graded algebra (for short, a $\ga$) $A^{\hdot}$ 
is \emph{connected}\/ if $A^0$ is 
the $\k$-span of the unit $1$ (and thus $A^0=\k$).  If $A$ is a connected 
$\dga$, then clearly its cohomology algebra, $H^{\hdot}(A)$, is again connected.  

This terminology is inspired by algebraic topology, where a 
continuous map $f\colon X\to Y$ between two topological 
spaces is said to be $q$-connected if it induces isomorphisms 
on homotopy groups up to degree $q$ and an epimorphism 
in degree $q+1$.  By Hurewicz's theorem, the induced map 
in cohomology, $H^{\hdot}(f)\colon H^{\hdot} (Y) \to H^{\hdot} (X)$, 
is also $q$-connected.  Note that the map $X\to \{\text{pt}\}$ 
is $q$-connected if and only if the space $X$ is $q$-connected. 

Finally, let $\Cat$ be a subcategory of the category 
of cochain complexes, for instance, $\CDGA$ or the category of differential graded Lie algebras, 
$\DGL$.  Two objects in this category, $C$ and $C'$, have the same {\em $q$-type}\/ 
(denoted $C\simeq_q C'$) if they can be connected in $\Cat$ 
by a zig-zag of $q$-equivalences.  

\subsection{CDGA models for spaces}
\label{subsec:models}

We now review the construction and some basic properties of $\dga$ models 
of spaces, following \cite{Su, Mo, Hal, Le, FHT, DP-ccm}.  
We will denote by $\Omega^{\hdot}_{\k}(X)$ 
Sullivan's de Rham algebra of a topological space $X$, constructed by using 
differential forms with $\k$-polynomial coefficients on standard simplices, see \cite{Su}.  
The resulting functor $\Tp \to \text{$\k$-$\CDGA$}$ has, among other things, the 
property that $H^{\hdot}(\Omega^{\hdot}_{\k}(X)) \cong H^{\hdot}(X,\k)$, 
as graded $\k$-algebras.

To define monodromy representations of flat connections (over $\k=\R$ or $\C$), 
we will also need the similar $\dga$ $\Omega^{\hdot}(X,\k)$, constructed from usual 
smooth $\k$-forms.   It is known that $\Omega^{\hdot}(X,\k)$ has the same $\infty$-type 
as the sub-$\dga$ $\Omega^{\hdot}_{\k}(X)$, in a natural way.  

Let $A$ be a $\dga$.  For $0\le q\le \infty$, we say that $A$ is 
a {\em $q$-model}\/ for the space $X$ if $\Omega_{\k}(X)\simeq_q A$.  
We also say that $X$ is {\em $q$-finite}\/  if it has the homotopy type 
of a connected CW-complex with finite $q$-skeleton.  Similarly, 
we say that $A$ is $q$-finite if it is connected and $\dim \bigoplus_{i\le q} A^i<\infty$. 
Once again, we shall omit $q$ from the notation when $q=\infty$. 

The category $\ACDGA$ of {\em augmented}, commutative differential graded 
algebras has objects $(A,\varepsilon)$, where the augmentation map 
$\varepsilon \colon A\to \k$ is a morphism of $\dga$s, while the morphisms 
in this category are the $\cdga$ maps commuting with augmentations. 
When $X$ is a pointed space, both $\Omega^{\hdot}(X,\k)$ and 
$\Omega^{\hdot}_{\k}(X)$ become $\ACDGA$s, again in a natural way. 

A connected  $\dga$~$A$ has a unique augmentation map, sending $A^{+}$ 
to $0$, and the unit to $1$.  Moreover, for every augmented $\dga$~$A'$, 
we have that 
\begin{equation}
\label{eq:homacdga}
\Hom_{\ACDGA} (A,A') = \Hom_{\CDGA} (A,A') .
\end{equation}

\subsection{Hirsch extensions and relative minimal models}
\label{subsec:minmod}

Let $U^{\hdot}=\bigoplus_{i\ge 1} U^i$ be a positively graded $\k$-vector space.  The 
{\em free commutative graded algebra}\/ on $U$, denoted by $\bwedge U^{\hdot}$, 
is the tensor product of the symmetric graded algebra on $U^{\text{even}}$ and 
the exterior graded algebra on $U^{\text{odd}}$.  We say that a $\dga$ is 
free if the underlying $\ga$ has this property. 
Since $\bwedge U^{\hdot}$ is connected, it has a unique augmentation, 
denoted by $\varepsilon_U$. 

Let $A=(A^{\hdot},d_A)$ be a $\dga$, and denote by $Z^{\hdot}(A)$ the graded 
vector space of cocycles.  Given a finite-dimensional graded vector space $U^{\hdot}$,  
and a degree $1$ linear map, $\tau\colon U^{\hdot} \to Z^{\hdot+1}(A)$, we denote by 
$(A\otimes_{\tau} \bwedge U , d)$ the corresponding {\em Hirsch extension}.  By 
definition, this is the $\dga$ whose underlying $\ga$ is  $A^{\hdot}\otimes \bwedge U^{\hdot}$, 
and whose differential restricts to $d_A$ on $A$ and to $\tau$ on $U$.  
If $A$ is an $\adga$ with augmentation $\varepsilon_A$, then 
$A\otimes_{\tau} \bwedge U$ is also an $\adga$, with augmentation 
$\varepsilon_A\otimes \varepsilon _U$.  The Hirsch extension depends only 
on the map $[\tau]\colon U^{\hdot} \to H^{\hdot+1}(A)$, in the following 
sense:  if $[\tau]=[\tau']$, then $A\otimes_{\tau} \bwedge U \cong 
A\otimes_{\tau'} \bwedge U$ in $\CDGA$, via an isomorphism extending $\id_A$.  
When $\dim U=1$, we speak of an `elementary' Hirsch extension. 

A {\em relative Sullivan algebra}\/ with base $B$ is a direct limit 
of elementary Hirsch extensions, starting from the $\dga$ $B$. 
When the base is $\k$, concentrated in degree $0$, we simply speak 
of a Sullivan algebra.  Such a $\dga$ is necessarily of the form 
$\M=(\bwedge U^{\hdot}, d)$.  If $\im(d)\subseteq \bwedge^{\ge 2} U$, 
we say $\M$ is a {\em minimal}\/ Sullivan algebra. If, moreover, all 
Hirsch extensions have degree at most $q$, the $\dga$ $\M$ is 
said to be {\em $q$-minimal}. 

A {\em $q$-minimal model map}\/ for a $\dga$ $A$ is a $q$-equivalence 
$\rho\colon \M_q\to A$, with $\M_q$ a $q$-minimal Sullivan algebra. 
Any $\dga$ $A$ whose cohomology algebra is connected admits 
a $q$-minimal model map. 
If $\rho'\colon \M'_q\to A$ is another $q$-minimal model map for $A$, 
then $\M'_q$ and $\M_q$ are isomorphic in $\CDGA$. 
Consequently, if $A$ and $\overline{A}$ are two $\dga$s with connected homology, then 
$A\simeq_{q} \overline{A}$ in $\CDGA$ if and only if there is 
a $q$-minimal $\dga$ $\M_q$, and a short zig-zag of 
$q$-equivalences in $\CDGA$ of the form 
\begin{equation}
\label{eq:shortzig}
\xymatrix{A & \mathcal{M}_q \ar_{\rho}[l]  \ar^{\bar\rho}[r] & \overline{A}}. 
\end{equation}

Recall that a relative Sullivan algebra with base $B$ is 
a $\dga$ of the form $A=(B \otimes \bigwedge U, d)$.  
When $B$ is an augmented algebra, with augmentation ideal 
$\widetilde{B}:= \ker (\varepsilon_B)$, 
the quotient $\dga$, $A/(\widetilde{B} \otimes \bigwedge U)=(\bigwedge U,\bar{d})$, is 
called the {\em fiber}\/ of $A$. Following \cite{Hal, FHT}, 
we say that $A$ is a {\em minimal Sullivan algebra in the relative sense}\/ 
if the fiber is minimal. Allowing also degree $0$ Hirsch extensions, we may speak 
of {\em weak}\/ relative Sullivan (minimal) algebras. 

Let $\Phi\colon B\to C$ be a $\cdga$ map, and assume $B$ is augmented 
and $H^{\hdot}(B)$ is connected. A {\em relative minimal model map}\/ for $\Phi$ 
is an $\infty$-equivalence of $\dga$s, 
$h\colon \M\to C$, where $\M=(B \otimes \bigwedge U, d)$ 
is a relative minimal Sullivan algebra and $\left. h \right|_B= \Phi$.
If $\Phi$ is a $0$-equivalence, then $\Phi$ admits a relative minimal model map; 
moreover, any two such maps, $h$ and $h'$, have isomorphic fibers, 
see \cite[\S14]{FHT}.  In fact, existence and uniqueness also hold in 
the weak sense, assuming only that $H^0(\Phi)$ is an isomorphism, 
see \cite[Ch.~6]{Hal}.  In particular, if $\Phi\colon B\to C$ is a $0$-equivalence 
and $h\colon \M\to C$ is any relative minimal model map for $\Phi$ in the 
weak sense, then necessarily the fiber of $\M$ is connected. Hence, if 
$\Phi$ is an $\acdga$ map, then $\M$ is canonically augmented and 
both $h$ and the inclusion $B\inj \M$ preserve augmentations. 

\subsection{Homotopies and equivalences}
\label{subsec:homeq} 

Let $\bwedge (t, dt)$ be the free, 
contractible $\dga$ generated by $t$ in degree $0$ and $dt$ in degree $1$. 
For each $s\in \k$, let $\ev_s\colon  \bwedge (t, dt) \to \k$ be the 
$\dga$ map sending $t$ to $s$ and $dt$ to $0$.  This 
map induces another $\dga$ map,
\begin{equation}
\label{eq:Evs dga}
\Ev_s = \id\otimes \ev_s \colon A \otimes \bwedge (t, dt) \to A \otimes \k =A. 
\end{equation}

\begin{definition}
Two $\dga$ maps 
$\psi_0, \psi_1\colon A\to A'$ are said to be {\em homotopic}\/ 
(in $\CDGA$)
if there is a $\dga$ map $\Psi\colon A\to A' \otimes \bwedge (t, dt)$ 
such that $\Ev_s\circ \Psi = \psi_s$ for $s=0,1$.  Likewise,  
two $\adga$ maps $\psi_0$ and $\psi_1$ as above are homotopic 
(in $\ACDGA$) if the homotopy $\Psi$ also satisfies $\Ev'\circ \Psi =\varepsilon$, 
where $\Ev'$ denotes the $\dga$ map $\varepsilon' \otimes \id \colon 
A' \otimes \bwedge (t, dt) \to  \bwedge (t, dt)$.  
\end{definition}

Plainly, equality of maps 
implies homotopy, in both categories. Note that augmented homotopy 
is strictly stronger than homotopy. Another useful remark is that homotopic maps 
in $\CDGA$ induce the same map in cohomology. 

Denote by $\ACDGA_0$ the full subcategory of $\ACDGA$ whose 
objects have connected cohomology.  Fix an integer $q\ge 0$.

\begin{definition}
\label{def:qeqmaps}
An {\em elementary $q$-equivalence}\/ in $\ACDGA_0$ between two $\ACDGA_0$-morphisms
$\Phi_0\colon A_0' \to A_0$ and $\Phi_1\colon A_1' \to A_1$ consists of 
two $\acdga$ maps, $\psi\colon A_1\to A_0$ and $\psi' \colon A'_1\to A'_0$, 
both of which are $q$-equivalences, and such that $\psi\circ \Phi_1$ is 
homotopic to $\Phi_0\circ \psi'$
in $\ACDGA$. We denote by $\simeq_q$ the associated 
equivalence relation between morphisms in $\ACDGA_0$. 
\end{definition}

In other words, if $\Phi\colon A' \to A$ and $\bar{\Phi}\colon B' \to B$ are 
two such morphisms, we say that $\Phi\simeq_q \bar{\Phi}$ in $\ACDGA_0$ if there 
are two zig-zags, $Z$ and $Z'$, of $q$-equivalences in $\ACDGA$, and $\adga$ maps
$\Phi_1, \dots, \Phi_{\ell -1}$ such that the following diagram commutes, 
up to augmented homotopy:
\begin{equation}
\label{eq:ziggy}
\begin{gathered}
\xymatrix{
Z: \hspace{-20pt} & A  & A_1 \ar_(.45){\psi_0}[l]  \ar^{\psi_1}[r] & \cdots 
& A_{\ell-1}   \ar[l]\ar^{\psi_{\ell-1}}[r] & B \, \phantom{.}
\\
Z': \hspace{-20pt} & A'  \ar^{\Phi}[u] & A'_1 
 \ar^{\Phi_1}[u] \ar_(.45){\psi'_0}[l]  \ar^{\psi'_1}[r] & \cdots 
& A'_{\ell-1}   \ar_{\Phi_{\ell-1}}[u] \ar[l]\ar^{\psi'_{\ell-1}}[r] & B' \, .  
\ar_{\bar{\Phi}}[u]
}
\end{gathered}
\end{equation}

Forgetting augmentations in Definition \ref{def:qeqmaps}, we obtain 
the equivalence relation $\simeq_q$ between maps in $\CDGA_0$. 
When $q=\infty$, we will simply write this as $\Phi\simeq \bar\Phi$. 

\begin{lemma}
\label{lem:zig}
Let $A$ and $\overline{A}$ be two $\adga$s. Assume $H^{\hdot}(A)$ is connected, 
and let $\rho \colon \mathcal{M}_q \to A$ be a $q$-minimal $\dga$ model map as above. 
Then $A\simeq_q \overline{A}$ in $\CDGA$ if and only if there is a short zig-zag 
of $q$-equivalences in $\ACDGA$ as in \eqref{eq:shortzig}.
\end{lemma}

\begin{proof}
By the discussion from \S\ref{subsec:minmod}, we have that 
$A\simeq_q \overline{A}$ in $\CDGA$ if and only if $A$ and $\overline{A}$  
share the same $q$-minimal model $\mathcal{M}_q$.  The fact that $\mathcal{M}_q$ 
is connected takes care of the augmentations.  
\end{proof}

\subsection{Formal spaces and maps}
\label{subsec:fmaps}

We conclude this section with some formality notions, 
for both spaces and maps, as well as models thereof.  To start with, we say 
that a $\dga$ $A$ is {\em $q$-formal}\/ if $(A^{\hdot},d)\simeq_q (H^{\hdot} (A),d=0)$ 
in $\CDGA$. Clearly, $q$-formality implies $p$-formality, for all $p\le q$. 
By definition, a space $X$ is $q$-formal over $\k$ if $\Omega_{\k}(X)$ 
has this property.   A $q$-finite, $q$-formal space $X$ has the $q$-finite 
$q$-model $(H^{\hdot} (X),d=0)$.  As before, we will mostly omit 
$q$ from notation when $q=\infty$. 
Compact \Ka manifolds are well-known to be formal, by the main result 
from \cite{DGMS}. 

Following \cite{Su, DGMS, VP}, we say that a morphism $\Phi\colon A'\to A$ 
in $\CDGA_0$ is {\em formal}\/ if there is a diagram consisting of two 
elementary equivalences in $\CDGA_0$, 
\begin{equation}
\label{eq:phif}
\begin{gathered}
\xymatrixcolsep{30pt}
\xymatrix{
A  &\M \ar_(.45){\psi}[l]  \ar^(.35){\bar\psi}[r] & (H^{\hdot}(A), d=0)\,\,
\\
A' \ar^{\Phi}[u] &\M' \ar_(.45){\psi'}[l] \ar^{\widehat{\Phi}}[u] \ar^(.35){\bar\psi'}[r] 
& (H^{\hdot}(A'), d=0)\,, \ar_{H^{\hdot}(\Phi)}[u]
}
\end{gathered}
\end{equation}
such that 
both $\M$ and $\M'$ are minimal Sullivan algebras.  Furthermore, 
we say that a continuous map $f\colon X\to X'$ is formal 
(over $\k$) if the induced morphism between Sullivan de Rham models,  
$\Omega_{\k}(f)\colon \Omega_{\k}(X')\to 
\Omega_{\k}(X)$, has this property. 

\begin{prop}
\label{prop:mainf}
Let $f\colon X\to X'$ be a continuous map between pointed spaces. 
Assume that $f$ is formal over $\k$, and $H^1(f)$ is injective. Then 
$\Omega(f)\simeq H^{\hdot}(f)$ in $\k$-$\ACDGA_0$. 
\end{prop}

\begin{proof}
In \cite[II.3]{VP}, Vigu\'e-Poirrier uses the formality of the map 
$f$ to construct a commuting diagram in $\CDGA$ of the form 
\begin{equation}
\label{eq:vp}
\begin{gathered}
\xymatrixcolsep{30pt}
\xymatrix{
\Omega(X)  &\bigwedge U' \otimes \bigwedge U 
\ar_(.55){\psi}[l]  \ar^(.47){\bar\psi}[r] & (H^{\hdot}(X), d=0)
\\
\Omega(X')  \ar^{\Omega(f)}[u] &\bigwedge U'\ar_(.55){\psi'}[l] 
\ar@{^{(}->}^{j}[u] \ar^(.39){\bar\psi'}[r] 
& (H^{\hdot}(X'), d=0) \ar_{H^{\hdot}(f)}[u]
}
\end{gathered}
\end{equation}
where all horizontal arrows are $\infty$-equivalences and $j$ is the 
canonical inclusion. Moreover, $\psi'$ and $\bar\psi'$ are 
minimal model maps (in particular, $\bigwedge U'$ is connected), 
and $\bigwedge U' \otimes \bigwedge U$ is a relative minimal 
Sullivan algebra in the weak sense, with base $\bigwedge U'$.

The injectivity assumption on $H^1(f)$ implies that $H^{\hdot}(f)\circ \bar\psi'$ 
is a $0$-equivalence.  From the discussion in \S\ref{subsec:minmod}, we 
deduce that $U^0=0$. This shows that  
$\bigwedge U' \otimes \bigwedge U$ is also connected.  
Hence, \eqref{eq:vp} is a commuting diagram 
in $\ACDGA_0$, and our claim follows.
\end{proof}

\section{Deformation theory of representation varieties}
\label{sect:defrefp}

Following Goldman--Millson \cite{GM} and Manetti \cite{Ma},  we recall, in 
a convenient form, two basic properties of the deformation functor associated to a 
differential graded Lie algebra ($\dgl$ for short).  We then apply these techniques 
to the representation varieties of discrete groups.  

\subsection{Deformation functors}
\label{subsec:dfun}

We denote by $\Art$ the category of Artinian local $\k$-algebras. 
Given such an algebra $(\art,\m_{\art})$ and a differential graded Lie algebra 
$L$, we consider the (nilpotent) $\dgl$~$L\otimes \m_{\art}$, and 
the set of solutions (called {\em flat connections}) 
to the Maurer--Cartan equation, 
\begin{equation}
\label{eq:mc}
\F(L\otimes \m_{\art}) = \{ \omega \in L^1\otimes \m_{\art} \mid 
d\omega + \tfrac{1}{2} [\omega,\omega]=0\},
\end{equation}
with basepoint $0\in \F(L\otimes \m_{\art})$.   Clearly, this construction is 
bifunctorial.  Let 
\begin{equation}
\label{eq:gauge}
\G(L\otimes \m_{\art})=\exp (L^0 \otimes \m_{\art})
\end{equation}
be the (Campbell--Hausdorff) {\em gauge group}\/ of the nilpotent 
Lie algebra $L^0\otimes \m_{\art}$, with underlying set  $L^0\otimes \m_{\art}$. 
This group acts bifunctorially 
on $\F(L\otimes \m_{\art}) $ by 
\begin{equation}
\label{eq:gauge action}
\exp(\alpha)\cdot \omega = \omega + \sum_{n=0}^{\infty} 
\frac{\ad(\alpha)^n}{(n+1)!} ([\alpha,\omega] - d\alpha).
\end{equation}

The {\em deformation functor}, $\Def_L\colon \Art  \to \Set$, 
is defined by 
\begin{equation}
\label{eq:deff}
\Def_L(\art) =  \F(L\otimes \m_{\art})/\G(L\otimes \m_{\art}).
\end{equation}
It is readily seen that every $\dgl$-morphism $\psi\colon L\to L'$ induces 
a natural transformation, $\Def_{\psi} \colon \Def_L \to \Def_{L'}$.  

We now may state the Deligne--Schlesinger--Stasheff theorem, 
as recorded and proved in \cite[Thm.~2.4]{GM}. 

\begin{theorem}
\label{thm:1type}
If $L\simeq_1 L'$ in $\DGL$, then the deformation functors 
$\Def_L$ and  $\Def_{L'}$ are naturally isomorphic.  
\end{theorem}

\subsection{Homotopy invariance}
\label{subsec:hinv}

The homotopy relation in $\DGL$ takes the following form.
Given a $\dgl$ $L$, let us form the $\dgl$ $L\otimes \bwedge (t, dt)$, 
endowed with the canonical tensor product structure. 
For each $s\in \k$, we defined in \S\ref{subsec:homeq}
an evaluation $\dga$ map, $\ev_s\colon  \bwedge (t, dt) \to \k$, 
which sends $t\mapsto s$ and $dt\mapsto 0$.
Proceeding as before, we extend this map to a $\dgl$ map, 
$\Ev_s = \id\otimes \ev_s \colon L \otimes \bwedge (t, dt) \to L \otimes \k =L$.

Two $\dgl$ maps $\psi_0, \psi_1\colon L\to L'$ are said to be {\em homotopic}\/ 
if there is a $\dgl$ map $\Psi\colon L\to L' \otimes \bwedge (t, dt)$ such that 
$\Ev_s\circ \Psi = \psi_s$ for $s=0,1$. 
The notion of homotopy between $\dgl$ maps is related to deformation 
functors via the following basic result of Manetti \cite[Thm.~5.5]{Ma}. 

\begin{theorem}
\label{thm:manetti}
Let   $L$  be a $\dgl$, and let $\art$ be a local Artin algebra. Two flat connections  
$\beta_0, \beta_1\in \F(L\otimes \m_{\art})$ are equal in $\Def_{L}(\art)$ 
if and only if there is a flat connection 
$\omega\in \F(L\otimes \bwedge (t, dt)\otimes \m_{\art})$ 
such that $(\Ev_s \otimes \id) \omega = \beta_s$ for $s=0,1$.
\end{theorem}

This theorem has an immediate corollary, which will be useful in the sequel.

\begin{corollary}
\label{cor:man}
If $\psi_0, \psi_1\colon L\to L'$ are homotopic in $\DGL$, then  
$\Def_{\psi_0}=\Def_{\psi_1}$. 
\end{corollary}

\subsection{Deformation theory of $\dga$s}
\label{subsec:defdga}
We consider now the bifunctor $\CDGA \times \Lie \to \DGL$ 
which associates to a $\dga$~$A$ and a Lie algebra $\g$ the $\dgl$
\begin{equation}
\label{eq:dgl lag}
L = A\otimes \g,
\end{equation}
endowed with the canonical tensor product structure $[a\otimes g ,
a'\otimes g'] = aa' \otimes [g,g']$ and differential 
$\partial (a\otimes g) = da \otimes g$. 

We will also need an augmented version of this construction. 
Given an augmented $\dga$~$(A,\varepsilon)$ and a Lie algebra $\g$, 
we denote by $\widetilde{A}=\ker(\varepsilon)$ the augmentation differential ideal, 
and we consider the sub-$\dgl$ $\widetilde{L} = \widetilde{A}\otimes \g$ 
of the $\dgl$ $L = A\otimes \g$. This construction is again bifunctorial. 

\begin{remark}
\label{rem:tilde}
Given a $q$-equivalence $\psi\in \Hom_{\ACDGA} (A,A')$, it is easy 
to check that the induced maps, $\tilde{\psi}\colon \widetilde{A} \to \widetilde{A}'$ 
and $\tilde{\psi}\otimes \id\colon \widetilde{L} \to \widetilde{L}'$, are again 
$q$-equivalences, provided that both $H^{\hdot}(A)$ and $H^{\hdot}(A')$ 
are connected.  Consequently, if  $A\simeq_q A'$ in $\ACDGA_0$ 
then $\widetilde{L} \simeq_q \widetilde{L}'$ in $\DGL$.
\end{remark}

Let $\g$ be a Lie algebra.    The proof of the next lemma is straightforward.

\begin{lemma}
\label{lem:htpy}
A $\CDGA$ homotopy, 
$\Psi\colon A\to A' \otimes \bwedge (t, dt)$,   
between two maps, $\psi_0$ and $\psi_1$, induces a $\DGL$ homotopy, 
$\Psi\otimes \id \colon A\otimes \g \to A' \otimes \g \otimes \bwedge (t, dt)$, between 
the maps $\psi_0\otimes  \id$ and $\psi_1\otimes \id$.  Moreover, 
if $\Psi$ is an augmented homotopy, then $\Psi$ induces a $\DGL$ homotopy, 
$\widetilde{\Psi}\otimes \id \colon \widetilde{A}\otimes \g \to 
\widetilde{A}' \otimes \g \otimes \bwedge (t, dt)$, between 
the maps $\tilde\psi_0\otimes  \id$ and $\tilde\psi_1\otimes \id$.  
\end{lemma}

\subsection{Deformation theory of augmented $\dga$s}
\label{subsec:deg adga}
Our next goal is to relate $q$-types of $\dga$s to the deformation theory 
of $\adga$s. 
Fix $q\ge 1$, and let $Z$ be a zig-zag of $q$-equivalences in $\ACDGA$,
\begin{equation}
\label{eq:defz}
\xymatrix{
A_0 & A_1 \ar_(.45){\psi_0}[l]  \ar^{\psi_1}[r] & \cdots 
& A_{\ell-1}   \ar[l]\ar^{\psi_{\ell-1}}[r] & A_{\ell}
},
\end{equation}
where $H^0(A_0)=\k\cdot 1$.  
By Remark \ref{rem:tilde} and Theorem \ref{thm:1type}, the zig-zag $Z$ induces a 
natural bijection 
\begin{equation}
\label{eq:betaz}
\xymatrix{
\beta_Z\colon \F(\widetilde{A}_{\ell}\otimes \g \otimes \m_{\art}) / 
\G (\widetilde{A}_{\ell}\otimes \g \otimes \m_{\art}) 
\ar^(.51){\simeq}[r] & \F(\widetilde{A}_0\otimes \g \otimes \m_{\art}) / 
\G (\widetilde{A}_0\otimes \g \otimes \m_{\art}) 
}
\end{equation}
for all local Artin algebras $\art$. It is important to note that, if $Z$ and $Z'$ 
are two different zig-zags of $q$-equivalences connecting $A_0$ to $A_{\ell}$, 
then the bijections $\beta_Z$ and $\beta_{Z'}$ may also be different. 

\begin{prop}
\label{prop:ztos}
Let $\rho\colon \mathcal{M}_q\to A_0$ be a $q$-minimal model map. 
There is then a short zig-zag of $q$-equivalences in $\ACDGA$,
\[
\xymatrix{S\colon A_0 & \mathcal{M}_q \ar_(.42){\rho}[l]  \ar^{\bar\rho}[r] & A_{\ell}
},
\]
such that $\beta_Z=\beta_{S}$.
\end{prop}

\begin{proof}
We will construct, by induction on $0\le i \le \ell$, a collection of $q$-minimal model 
maps $\rho_i\colon \mathcal{M}_q \to A_i$, which form, together with the maps 
$\psi_i$ from \eqref{eq:defz}, homotopy-com\-mutative triangles in $\ACDGA$,  
starting with $\rho_0=\rho$.  Once this done, we set $\bar\rho=\rho_{\ell}$.  The equality   
$\beta_Z=\beta_{S}$ then follows from Lemma \ref{lem:htpy} and Corollary \ref{cor:man}.

For the induction step, we first assume that $\psi_i\colon A_i\to A_{i+1}$.  
Then we take $\rho_{i+1} = \psi_i\circ \rho_i$.  Finally, suppose that 
$\psi_i\colon A_{i+1}\to A_{i}$. The lifting property up to homotopy for 
$\dga$ maps also holds for $\adga$ maps, and implies that we may find 
a $\dga$ map $\rho_{i+1}\colon \mathcal{M}_q \to A_{i+1}$ such that 
$\psi_i\circ \rho_{i+1}$ is homotopic to $\rho_i$ in $\ACDGA$.  The 
fact that $\rho_{i+1}$ must be a $q$-equivalence is easily checked, 
thereby completing the proof.
\end{proof}

\subsection{Representation varieties}
\label{subsec:rep vars}   

Let $\pi$ be a discrete group, and let $G$ be a $\k$-linear algebraic group. 
The set $\Hom(\pi, G)$ of group homomorphisms from $\pi$ to $G$ has a  
natural structure of an affine scheme.  This set depends bi-functorially on 
$\pi$ and $G$, and has a natural base point, the  trivial representation, $1$. 
Furthermore, $G$ acts by conjugation on $\Hom(\pi, G)$.

Now suppose $\pi$ is a finitely generated group.  (Note that the 
fundamental group $\pi=\pi_1(X,x)$ of a pointed CW-space is finitely 
generated if and only if $X$ is $1$-finite.) In this case, the set $\Hom(\pi, G)$ 
has a natural structure of affine variety, called the {\em $G$-representation 
variety}\/ of $\pi$.   Moreover, every homomorphism $\varphi\colon \pi \to \pi'$ 
induces an algebraic morphism between the corresponding representation varieties,
$\varphi^{!}\colon \Hom (\pi', G) \to\Hom (\pi, G)$.  We will come back to this  
point in Lemma \ref{lem:zeronto-top}. 

Clearly, the $G$-representation variety of the free group $F_n$ 
is equal to the $n$-fold direct product $G^{n}$.  Much is known 
about the varieties of commuting matrices, 
for instance, that $\Hom(\Z^2,\GL_n(\C))$ is irreducible.  
Nevertheless, many open questions remain about 
the precise structure of the varieties $\Hom(\Z^n,G)$, see 
for instance \cite{BFM, AC} and references therein.   
Perhaps the most-studied family of representation 
varieties is that of fundamental groups of closed 
orientable surfaces $\Sigma_g$.  For instance, 
it is known that $\Hom(\pi_1(\Sigma_g),G)$ is connected 
if $G=\SL_n(\C)$, and an absolutely irreducible and $\Q$-rational variety 
if $G=\GL_n(\C)$, see \cite{Go, RBC}.

\subsection{Flat connections}
\label{subsec:flat conn}   

The infinitesimal counterpart to the representation varieties is provided by 
the space of flat connections. 
Given a $\dga$ $A$ and a Lie algebra $\g$, we will denote by $\F(A,\g)$ 
the set of flat connections on the $\dgl$ $A\otimes \g$. This set behaves 
bi-functorially, and has a natural basepoint, the trivial flat connection $0$. 
For a local Artin $\k$-algebra $\art$, the gauge group
\begin{equation}
\label{eq:gauge-bis}
\G(A\otimes\g\otimes \m_{\art})=\exp (A^0 \otimes\g\otimes \m_{\art})
\end{equation}
acts naturally on $\F(A\otimes\g\otimes \m_{\art})$.  If $A$ is an augmented 
$\dga$, we have that $\F(\widetilde{A}\otimes\g\otimes \m_{\art})=\F(A\otimes\g\otimes \m_{\art})$ 
and $\G(\widetilde{A}\otimes\g\otimes \m_{\art})\subseteq \G(A\otimes\g\otimes \m_{\art})$, 
with the augmented gauge group acting by restriction. In the particular case when 
$A$ is connected, the augmented gauge group is trivial, and we obtain a natural 
identification,
\begin{equation}
\label{eq:flatdef}
\F(A, \g\otimes \m_{\art})=\Def_{\widetilde{A}\otimes \g} (\art).
\end{equation}

If both $A^1$ and $\g$ are finite-dimensional, then the set 
$\F(A,\g)$ has a natural structure of affine variety, which 
we shall call the {\em $\g$-variety of flat connections}\/ on 
the $\dga$ $A$.

Now let $(X,x)$ be a pointed space with fundamental group $\pi$, 
and let $G$ be a linear algebraic group over $\k=\R$ or $\C$, 
with Lie algebra $\g$. The monodromy construction from 
\cite[\S6.3]{DP-ccm} gives a map 
\begin{equation}
\label{eq:mon}
\xymatrix{\mon\colon \F(\Omega(X,\k), \g) \ar[r]& \Hom(\pi,G)}
\end{equation}
which extends the classical monodromy map for smooth manifolds, 
and has nice naturality properties.  Furthermore, for each local Artin $\k$-algebra 
$\art$, we have a natural monodromy map 
\begin{equation}
\label{eq:mon-bis}
\xymatrix{\mon\colon \F(\Omega(X,\k), \g \otimes \m_{\art})  \ar[r]&
\Hom(\pi,\exp (\g \otimes \m_{\art}))}.
\end{equation}

The equivariance property of the monodromy map for smooth 
manifolds described in \cite[(5-8)]{GM} can be extended 
to arbitrary topological spaces, as follows.  

\begin{lemma}
\label{lem:moneq}
For any gauge equivalence $a\in \G(\Omega(X,\k) \otimes \g \otimes \m_{\art})$, 
we have a commuting diagram,
\[
\xymatrixcolsep{28pt}
\xymatrix{
\F(\Omega(X,\k), \g \otimes \m_{\art})  \ar^(.48){\mon}[r] \ar^{a}[d] &
\Hom(\pi,\exp (\g \otimes \m_{\art})) \ar^{c_a}[d]\\
\F(\Omega(X,\k), \g \otimes \m_{\art})  \ar^(.48){\mon}[r]&
\Hom(\pi,\exp (\g \otimes \m_{\art})),
}
\]
where $c_a$ stands for the conjugation action by $-(\varepsilon \otimes \id)(a)$
and $\varepsilon \otimes \id\colon \Omega^0(X,\k)\otimes \g \otimes \m_{\art}  \to 
 \g \otimes \m_{\art} $ is given by the augmentation $\varepsilon$ of $\Omega(X,\k)$ 
corresponding to the basepoint $x$. 
Consequently, the monodromy map factors through the action of 
the augmented gauge group. 
\end{lemma}

We will repeatedly work under the assumptions of Theorem B from \cite{DP-ccm}. 
Namely, we fix an integer $q\ge 1$, and we let $X$ be a pointed, $q$-finite space with 
fundamental group $\pi$. Next, we assume there is a $q$-finite $\dga$ $A$ such 
that $\Omega(X,\k)\simeq_q A$ in $\CDGA$.  Finally, we let $G$ be a 
linear algebraic group over $\k=\R$ or $\C$, with Lie algebra $\g$.  

Now let $\rho_1\colon \mathcal{N}\to \Omega(X,\k)$ be a 
`$\pi_1$-adapted' $1$-minimal model map, as in \cite[\S6.4]{DP-ccm}.  
By minimal model theory of $\dga$s, we may extend $\rho_1$ to a 
$q$-minimal model map, $\rho_q\colon \mathcal{M}_q \to \Omega(X,\k)$. 
By Lemma \ref{lem:zig}, we may find a zig-zag of $q$-equivalences 
in $\ACDGA$ of the form 
\begin{equation}
\label{eq:zigspec}
S\colon \xymatrix{
 \Omega(X,\k) & \mathcal{M}_q \ar_(.36){\rho_q}[l]  \ar^{\bar\rho_q}[r] & A
}
\end{equation}
which fits into the basic setup from  \cite[\S7.2]{DP-ccm}.  We will 
call such zig-zag {\em special}. 

The next result is a topological analog of Theorem 6.8 from \cite{GM}, 
proved only for smooth manifolds. 

\begin{theorem}
\label{thm:monbij}
Let $X$ be a $1$-finite space. Then the natural map
\[
\xymatrix{\mon\colon \F(\Omega(X,\k)\otimes \g \otimes \m_{\art})/
\G(\widetilde{\Omega}(X,\k) \otimes \g \otimes \m_{\art})  \ar[r]&
\Hom(\pi,\exp (\g \otimes \m_{\art}))}
\]
from Lemma \ref{lem:moneq} is a bijection, for all local Artin $\k$-algebras $\art$.
\end{theorem}

\begin{proof}
Let $\rho_1\colon \mathcal{N}\to \Omega(X,\k)$ be a $\pi_1$-adapted 
$1$-minimal model map.   By \cite[Prop.~6.16]{DP-ccm}, the composite
\[
\xymatrixcolsep{28pt}
\xymatrix{
\F(\mathcal{N}\otimes \g \otimes \m_{\art}) \ar^(.48){\rho_1\otimes \id}[r] 
& \F(\Omega(X,\k)\otimes \g \otimes \m_{\art})  
\ar^{\mon}[r] & \Hom(\pi,\exp (\g \otimes \m_{\art}))
}
\]
is a bijection.  Since $\mathcal{N}$ is connected, formula \eqref{eq:flatdef} allows us to replace 
$\F(\mathcal{N} \otimes \g \otimes \m_{\art})$ by $\Def_{\widetilde{\mathcal{N}}\otimes \g} (\art)$. 
Using now Lemma \ref{lem:moneq}, we see that the above bijection is equal to the 
composite 
\[
\xymatrixcolsep{30pt}
\xymatrix{
\Def_{\widetilde{\mathcal{N}}\otimes \g} (\art) \ar^(.45){\tilde\rho_1\otimes \id}[r] 
&\Def_{\widetilde{\Omega}(X,\k)\otimes \g} (\art) \ar^(.42){\mon}[r] 
& \Hom(\pi,\exp (\g \otimes \m_{\art}))
}.
\]
Finally,  it follows from Theorem \ref{thm:1type} and Remark \ref{rem:tilde} 
that the map $\tilde\rho_1\otimes \id$ is also a bijection, and this completes 
the proof.
\end{proof}

Assume again that the hypotheses of Theorem B from \cite{DP-ccm} 
are satisfied.  Let $Z$ be a zig-zag of $q$-equivalences in $\ACDGA$ 
as in \eqref{eq:defz}, connecting $A_0=\Omega(X,\k)$ to $A_{\ell}=A$. 
Using Theorem \ref{thm:monbij} and formula \eqref{eq:flatdef}, 
we may then define a natural bijection 
\begin{equation}
\label{eq:alphaz}
\xymatrix{\alpha_Z := \mon \circ \beta_Z \colon \F(A, \g\otimes \m_{\art}) 
\ar^(.56){\simeq}[r]&  \Hom(\pi,\exp (\g \otimes \m_{\art})) }.
\end{equation}

\begin{corollary}
\label{cor:alphaspec}
For any zig-zag $Z$ as above, there is a special zig-zag $S$ 
such that  $\alpha_Z=\alpha_S$. 
\end{corollary}

\begin{proof}
Let $\rho_q\colon \mathcal{M}_q\to \Omega(X,\k)$ be a $q$-minimal model map 
extending a $\pi_1$-adapted $1$-minimal model map $\rho_1\colon \mathcal{N}\to \Omega(X,\k)$. 
By Proposition \ref{prop:ztos}, there is a special zig-zag $S$ as in diagram \eqref{eq:zigspec}  
such that $\beta_Z=\beta_S$.  The claim follows.
\end{proof}

\section{Cohomology jump loci and naturality properties}
\label{sect:cjl}

We now define two types of cohomology jump loci (one for spaces 
and the other for $\dga$s), and study some of the naturality 
properties these algebraic varieties enjoy. 

\subsection{Embedded cohomology jump loci}
\label{subsec:jumps}

Let $(X,x)$ be a pointed, path-connected space.  Set $\pi=\pi_1(X,x)$.  
For a $\k$-linear algebraic group $G$, the set $\Hom(\pi, G)$ is a 
parameter space for finite-dimensional local systems on $X$ of type $G$.   
When the space $X$ is $1$-finite (or, equivalently, when the group $\pi$ 
is finitely generated), this parameter space is an affine $\k$-variety. When $\k=\R$ or $\C$,
we let $\Hom(\pi, G)_{(1)}$ 
be the analytic germ at $1$ of this variety, and we denote by 
$R=R(\pi,G)$ the analytic local algebra of this germ. 

Given a $\cdga$ $A^{\hdot}$ and a Lie algebra $\g$, let  
$\F(A,\g)$ be the set of $\g$-valued flat connections on $A$. 
When both $A^1$ and $\g$ are finite-dimensional, this set is 
an affine variety.  We shall denote by $\overline{R}=\overline{R}(A,\g)$ 
the analytic local algebra of the germ $\F(A,\g)_{(0)}$.
Assume now that both $X$ and $A$ are $1$-finite, and 
that $\Omega_\k(X)\simeq_{1} A$ as $\dga$s.  Letting 
$\g$ be the Lie algebra of $G$, it then follows from 
\cite[Prop.~7.6]{DP-ccm} that the local algebras 
$R$ and $\overline{R}$ are isomorphic. 

Given a representation $\tau\colon \pi\to \GL(V)$, we let 
$V_\tau$ denote the local system on $X$ associated to $\tau$, 
that is, the left $\pi$-module $V$ defined by $g\cdot v=\tau(g) v$. 
Furthermore, we let $H^{\hdot}(X,V_\tau)$ be the twisted cohomology 
of $X$ with coefficients in this local system, see e.g.~\cite{Wh}.

\begin{definition}
\label{def:cv}
The {\em characteristic varieties}\/ of the space $X$ in degree $i\ge 0$ 
and depth $r\ge 0$ with respect to a  representation 
$\iota\colon G \to \GL(V)$ are the sets 
\[
\VV^i_r(X,\iota) =\{ \rho \in \Hom(\pi,G) \mid 
\dim_{\k} H^i(X, V_{\iota \circ \rho} )\ge r\}.
\]
\end{definition}

For each $i\ge 0$, the sequence $\{\VV^i_r(X,\iota)\}_{r\ge 0}$ is 
a descending filtration of $\Hom(\pi,G) = \VV^i_0(X,\iota)$. 
In the rank $1$ case, i.e., when $\iota$ is the canonical identification 
$\k^{\times}\to \GL_1(\k)$, we will drop the map $\iota$ from the notation, 
and simply write $\VV^i_r(X)$.   When $X=K(\pi,1)$ is a classifying space 
for the group $\pi$, we will denote the corresponding characteristic varieties 
by $\VV^i_r(\pi,\iota)$. 

We will refer to the pairs 
\begin{equation}
\label{eq:homcv}
\left( \Hom(\pi,G) , \VV^i_r(X,\iota) \right)
\end{equation}
as the (global) {\em embedded jump loci}\/ of $X$ with respect to $\iota$. 
Clearly, such pairs depend only on the homotopy 
type of $X$ and on the representation $\iota$.
If $\iota$ is a rational representation and $X$ is a $q$-finite space for some $q\ge 1$, 
then the sets $\VV^i_r(X,\iota)$ 
are closed subvarieties of the representation variety  $\Hom(\pi,G)$, for 
all $i\le q$ and $r\ge 0$; see \cite{DP-ccm, BW}.  

\subsection{Infinitesimal cohomology jump loci}
\label{subsec:resonance}

To define the infinitesimal counterpart of these loci, we start 
with a $\dga$ $A^{\hdot}$, a Lie algebra $\g$, and a 
representation $\theta\colon \g\to \gl(V)$. 
For each flat connection $\omega\in \F(A,\g)$, we turn the 
tensor product $A\otimes V$ into a cochain complex,
\begin{equation}
\label{eq:aomoto}
\xymatrixcolsep{22pt}
\xymatrix{(A\otimes V , d_{\omega})\colon  \
A^0 \otimes V \ar^(.65){d_{\omega}}[r] & A^1\otimes V
\ar^(.5){d_{\omega}}[r]
& A^2\otimes V   \ar^(.55){d_{\omega}}[r]& \cdots },
\end{equation}
using as differential the covariant derivative
$d_{\omega}=d\otimes \id_V + \ad_{\omega}$.  Here, if 
$\omega=\sum_i a_i \otimes g_i$,  with
$a_i\in A^1$ and $g_i\in \g$, then
$\ad_{\omega}(a\otimes v) = 
\sum_{i} a_i  a \otimes \theta(g_i)(v)$, 
for all $a \in A$ and $v\in V$.   It is readily checked that
the flatness condition on $\omega$ insures that $d_{\omega}^2=0$, 
see \cite{DP-ccm}.

\begin{definition}
\label{eq:resvar}
The {\em resonance varieties}\/ of the $\cdga$ $A^{\hdot}$ in 
degree $i\ge 0$ and depth $r\ge 0$ with respect to a representation 
$\theta\colon \g\to \gl(V)$ are the sets 
\begin{equation}
\label{eq:rra}
\RR^i_r(A, \theta)= \{\omega \in \F(A,\g)
\mid \dim_{\k} H^i(A \otimes V, d_{\omega}) \ge  r\}.
\end{equation}
\end{definition}

For each $i\ge 0$, the sequence $\{\RR^i_r(A,\theta)\}_{r\ge 0}$ is 
a descending filtration of $\F(A,\g) = \RR^i_0(A,\theta)$. 
In the rank one case, i.e., the case when $\theta$ is the 
canonical identification $\k\to \gl_1(\k)$, 
we will simply write $\RR^i_r(A)$ for the 
corresponding sets. 

We will refer to the pairs 
\begin{equation}
\label{eq:flatrv}
\left( \F(A,\g) , \RR^i_r(A, \theta) \right)
\end{equation}
as the (global) {\em infinitesimal embedded jump loci}\/ of 
$A$ with respect to $\theta$. 
If $A$ is $q$-finite for some $q\ge 1$, and both $\g$ and $V$ are finite-dimensional,
the sets $\RR^i_r(A, \theta)$ are closed subvarieties of
$\F(A,\g)$, for all $i\le q$ and $r\ge 0$; see \cite{DP-ccm, BW}.  

Assume now that both the space $X$ and the $\dga$ $A^{\hdot}$ 
are $q$-finite, for some $q\ge 1$, and that $\Omega_{\k}(X)\simeq_{q} A$ 
as $\dga$s. Let $\iota\colon G\to \GL(V)$ be a rational  
representation, and let $\theta \colon \g \to \gl(V)$ be its 
tangential representation.  As shown in \cite[Thm.~B]{DP-ccm}, 
there is then an analytic isomorphism $\F(A,\g)_{(0)} \isom \Hom(\pi,G)_{(1)}$ 
restricting to isomorphisms $\RR^i_r(A,\theta)_{(0)} \isom \VV^i_r(X,\iota)_{(1)}$ 
between the reduced analytic germs of the corresponding jump loci, for all $i\le q$ 
and $r\ge 0$. 

We aim in this section at also taking into account in this setting of continuous 
maps between pointed spaces and of augmented maps between their $q$-models. 
We start with a preliminary observation, which follows directly from the definitions. 
Namely, for all $i\le q$ and $r\ge 0$, 
\begin{equation}
\label{eq:simvoid} 
1\in \VV^i_r(X,\iota) \Leftrightarrow 0\in \RR^i_r(A,\theta) \Leftrightarrow 
b_i\cdot \dim(V) \ge r,
\end{equation}
where $b_i=b_i(X)=b_i(A)$ denotes the $i$-th (untwisted) Betti number. 

\subsection{Naturality properties of representation varieties}
\label{subsec:natj reps}

As mentioned previously, both ambient spaces for jump loci,
$\Hom(\pi,G)$ and $\F(A,\g)$, are bifunctorial.  On the other 
hand, for continuous and $\cdga$~maps, naturality of (global) 
jump loci requires certain connectivity hypotheses. 
To begin, we only assume the minimally required 
connectivity and finiteness conditions.  

\begin{lemma}
\label{lem:zeronto-top}
Let $f\colon (X,x)\to (X',x')$ be a $0$-connected, pointed map, 
and let $f_{\sharp}$ be the induced homomorphism on fundamental groups. 
Assume that $X$ is $1$-finite.  Then, for every linear algebraic group $G$, the 
morphism induced by $f_{\sharp}$ on representation varieties,
\begin{equation}
\label{eq:ztop}
\xymatrixcolsep{18pt}
\xymatrix{
f_{\sharp}^{!}\colon \Hom (\pi_1(X'), G) \ar[r]& \Hom (\pi_1(X), G)},
\end{equation}
is an isomorphism onto a closed subvariety.
\end{lemma}

\begin{proof}
Our $0$-connectivity assumption on $f$ means that the homomorphism 
$f_{\sharp}$ is surjective. Our $1$-finiteness assumption on $X$, then, 
implies that both fundamental groups are finitely generated. 
Let us present $\pi_1(X)$ as the quotient $F_m/R$ of a free group on 
$m$ generators, and then use the presentation for $\pi_1(X')$ 
induced by $f_{\sharp}$.  

By construction, the representation variety $\Hom (\pi_1(X), G)$ 
is the closed subvariety of $G^m$ defined by the equations given 
by the relators in $R$. The variety $\Hom (\pi_1(X'), G)$ sits also in
$G^m$, with the same defining equations as $\Hom (\pi_1(X), G)$, 
plus the equations coming from the lifts to $F_m$ of the elements of 
$\ker(f_{\sharp})$.  The claim readily follows.
\end{proof}

\subsection{Holonomy Lie algebras}
\label{subsec:holo}

Before proceeding, let us recall from \cite[\S4]{MPPS} the construction 
of the {\em holonomy Lie algebra}\/ $\h(A)$ of a $1$-finite $\dga$ $(A,d)$.
Set $A_i=(A^i)^*$, and let $\L(A_1)$ be the free Lie algebra on the dual 
vector space $A_1$.  We then define 
\begin{equation}
\label{eq:holo}
\h(A) := \L(A_1) / \ideal(\im(d^*+\cup^*)),
\end{equation}
where $d^*\colon A_2\to A_1=\L^1(A_1)$
and $\cup^*\colon A_2\to A_1\wedge A_1=\L^2(A_1)$ are the 
maps dual to the differential and the multiplication map in $A$, 
respectively.   This construction is functorial: if $\psi\colon A'\to A$
is a morphism of $1$-finite $\dga$s, then the linear map 
$\psi_1=(\psi^1)^*\colon A_1\to A'_1$
extends to a Lie algebra morphism $\L(\psi_1)\colon \L(A_1)\to \L(A'_1)$,
which in turn induces a Lie algebra morphism $\h(\psi)\colon \h(A)\to \h(A')$.
Finally, as shown in \cite[Prop.~4.5]{MPPS}, the canonical isomorphism
$A^1\otimes \g \isom  \Hom (A_1,\g)$ restricts to an identification
$\F(A,\g) \cong  \Hom_{\Lie} (\h(A), \g)$. 

\begin{lemma}
\label{lem:zeronto-inf}
Let $\psi\colon A'\to A$ be a $0$-connected $\dga$~map.   
Assume that $A$ is $1$-finite.  Then, for every finite dimensional 
Lie algebra $\g$, the morphism 
\begin{equation}
\label{eq:zalg}
\xymatrixcolsep{18pt}
\xymatrix{
\psi \otimes \id \colon \F(A', \g) \ar[r]& \F (A, \g)}
\end{equation}
is an isomorphism onto a closed subvariety.
\end{lemma}

\begin{proof}
Our $0$-connectivity assumption on $\psi$ means that 
both $A$ and $A'$ are connected $\dga$s, 
and that $\psi$ is injective in degree $1$.  
Our $1$-finiteness assumption on $A$, then, 
implies that $A'$ is also $1$-finite.
Furthermore, the injectivity of $\psi^1$ also implies 
that the map $\h(\psi)\colon \h(A)\to \h(A')$ is surjective. 

Using the above discussion, we may replace the affine map 
$\psi \otimes \id \colon \F(A', \g) \to \F (A, \g)$ between spaces 
of flat connections by the induced map 
\begin{equation}
\label{eq:holo shriek}
\xymatrixcolsep{18pt}
\xymatrix{
\h(\psi)^{!} \colon \Hom_{\Lie} (\h(A'), \g)  \ar[r]& \Hom_{\Lie} (\h(A), \g)}
\end{equation}
between representation varieties of Lie algebras.  The desired conclusion 
follows by the same argument as in Lemma \ref{lem:zeronto-top}, with 
groups replaced by Lie algebras.
\end{proof}

As we saw in the above proof, the $0$-connectivity of the $\dga$ map $\psi$ 
implies the surjectivity of the Lie algebra map $\h(\psi)$. The next example 
shows that the latter property is strictly weaker than the former.  Nevertheless, 
we chose to state the lemma the way we did, since higher connectivity properties 
for $\dga$ maps will be needed later on. 

\begin{example}
\label{ex:holonto}
Let $A$ be the cohomology ring of $S^1\vee S^1$, with trivial product 
and differential. Plainly, the holonomy Lie algebra $\h=\h(A)$ is the free 
Lie algebra on $2$ generators.  Let $\h/\Gamma_3 (\h)$ be the third nilpotent 
quotient of $\h$, and let $A'$ be the cochain $\dga$ of this nilpotent Lie algebra. 
Denote by $\psi \colon A'^1 \to A^1$ the dual of the composite $A_1 \to \h \to \h/\Gamma_3 \h$. 
It is not hard to check that $\psi$ extends to a morphism $\psi \colon A' \to A$ between 
finite $\dga$s, with the property that $\h(\psi)$ is surjective.  On the other hand, the 
map $\psi$ is {\em not}\/ $0$-connected, as can be seen by inspecting dimensions 
in degree $1$.
\end{example}

\subsection{Naturality properties of jump loci}
\label{subsec:natjumps}

We now turn to the naturality properties of embedded cohomology jump loci.

\begin{lemma}
\label{lem:natjump-top}
Let $f\colon (X,x)\to (X',x')$ be a $q$-connected, pointed map, 
and let $f_{\sharp}$ be the induced homomorphism on fundamental groups. 
Let $\iota\colon G\to \GL(V)$ be a  representation.  
Then the natural map 
\begin{equation}
\label{natv}
\xymatrixcolsep{18pt}
\xymatrix{
H^{\hdot}(f)\colon H^{\hdot}(X', V_{\iota \circ \rho'}) \ar[r]& 
H^{\hdot}(X, V_{\iota \circ \rho})
},
\end{equation}
where $\rho=f_{\sharp}^{!}(\rho')$ for $\rho'\in  \Hom (\pi_1(X'), G)$, is $q$-connected.
\end{lemma}

\begin{proof}
Without loss of generality, we may assume that $G=\GL(V)$ and $\iota$ is 
the identity map.   Using standard CW-approximation results from homotopy 
theory, as recounted for instance in \cite[Ch.~V]{Wh}, we may replace $f$, up to homotopy, 
by the inclusion of a CW-subcomplex $X$ into a CW-complex $X'$.  Since 
$f$ is assumed to be $q$-connected, $X'$ may be obtained by attaching cells 
of dimension at least $q+2$ to $X$.  

Using the long exact sequence in cohomology 
for the pair $(X',X)$, we see that our claim is equivalent to the vanishing of the 
twisted cohomology groups $H^i (X',X; V)$ for $i\le q+1$, for an arbitrary local 
system $V$ on $X'$.  Denote by $\{X_n'\}$ the relative skeletal filtration of $X'$. 
It is well-known that $H^{\hdot} (X',X; V)$ can be computed as the cohomology 
of the cellular twisted cochain complex, whose degree $n$ term is 
$H^{n} (X'_{n},X'_{n-1}; V)$, see e.g.~\cite[Ch.~VI]{Wh}. On the 
other hand, $X'_n=X'_{n-1}=X$ for $n\le q+1$, and this completes the proof. 
\end{proof}

\begin{lemma}
\label{lem:natjump-inf}
Let $\psi\colon A'\to A$ be a $q$-connected map in $\CDGA$, and  
let $\theta\colon \g\to \gl(V)$ be a  Lie algebra representation.  
Then the natural map 
\begin{equation}
\label{natr}
\xymatrixcolsep{18pt}
\xymatrix{
H^{\hdot}(\psi)\colon H^{\hdot}(A' \otimes V, d_{\omega'}) \ar[r]& 
H^{\hdot}(A \otimes V, d_{\omega})
},
\end{equation}
where $\omega=(\psi\otimes \id) (\omega')$ for $\omega' \in \F(A', \g)$ 
is $q$-connected.
\end{lemma}

\begin{proof}
Without loss of generality, we may assume that $\g=\gl(V)$ and $\theta$ 
is the identity map. Since $\psi$ is $q$-connected, the cochain map 
$\psi\otimes \id \colon (A' \otimes V, d_{\omega'})\to (A \otimes V, d_{\omega})$
is again $q$-connected.  The claim follows from Lemma 2.6 in \cite{MPPS} 
and its proof.  
\end{proof}

\begin{corollary}
\label{cor:jumpnat-top}
Let $f\colon X\to X'$ be a $(q-1)$-connected map between 
$q$-finite pointed spaces,  for some $q\ge 1$, 
and let $\iota\colon G\to \GL(V)$ be a rational  representation.
Then the induced morphism 
\begin{equation}
\label{eq:nattop}
\xymatrixcolsep{18pt}
\xymatrix{
f_{\sharp}^{!}\colon \Hom (\pi_1(X'), G) \ar[r]& \Hom (\pi_1(X), G)}
\end{equation}
is a closed embedding which induces isomorphisms 
$\VV^i_r(X',\iota) \to \VV^i_r(X,\iota) \cap \Hom (\pi_1(X'), G)$ 
for all $i< q$ and $r\ge 0$, and embeddings 
$\VV^q_r(X',\iota) \to \VV^q_r(X,\iota)$ for all $r\ge 0$.
\end{corollary}

\begin{proof}  
The fact that $f_{\sharp}^{!}$ is a closed embedding follows from Lemma \ref{lem:zeronto-top}. 
The other assertions are immediate consequences of Lemma \ref{lem:natjump-top}.
\end{proof}

If $f\colon \pi\surj\pi'$ is an epimorphism between finitely generated groups,  
the case $q=1$ from Corollary \ref{cor:jumpnat-top} implies that the morphism 
$f^{!}\colon \Hom(\pi',\k^{\times}) \to \Hom(\pi,\k^{\times})$ sends  
$\VV^1_1(\pi')$ into $\VV^1_1(\pi)$.  Without the $0$-connectivity (i.e., 
surjectivity) assumption on $f$, the conclusion may fail, as illustrated in the 
following simple example. 

\begin{example}
\label{ex:nonat}
Let $f\colon \Z \to F_2=\langle x,y\rangle $ be the inclusion sending 
$1$ to $x$.  Then $f^{!}\colon (\k^{\times})^2 \to \k^{\times}$ is the projection onto 
the first factor. On the other hand, $\VV^1_1(F_2)=(\k^{\times})^2$, whereas 
$\VV^1_1(\Z)=\{1\}$.
\end{example}

\begin{corollary}
\label{cor:jumpnat-inf}
Let $\psi\colon A'\to A$ be a $(q-1)$-connected map between $q$-finite $\dga$s
for some $q\ge 1$, and  
let $\theta\colon \g\to \gl(V)$ be a  Lie algebra representation, 
with $\g$ and $V$ finite-dimensional.   Then the natural morphism
\begin{equation}
\label{natinf}
\xymatrixcolsep{18pt}
\xymatrix{
\psi\otimes \id\colon  \F(A' , \g) \ar[r]&  \F(A , \g)
}
\end{equation}
is a closed embedding which induces isomorphisms 
$\RR^i_r(A',\theta) \to \RR^i_r(A,\theta) \cap \F(A' , \g)$ 
for all $i<q$ and $r\ge 0$, and embeddings 
$\RR^q_r(A',\theta) \to \RR^q_r(A,\theta)$ for all $r\ge 0$.
\end{corollary}

\begin{proof}  
The fact that $\psi\otimes \id$ is a closed embedding follows 
from Lemma \ref{lem:zeronto-inf}. The other assertions are 
immediate consequences of Lemma \ref{lem:natjump-inf}.
\end{proof}

\begin{corollary}
\label{cor:pijump}
Let $X$ be a pointed space with fundamental group $\pi$, 
let $f\colon X\to K:=K(\pi,1)$ be a classifying map, 
and let $\iota\colon G\to \GL(V)$ be a representation. 
Then the induced isomorphism  
$f_{\sharp}^{!}\colon  \Hom(\pi_1(K),G) \to \Hom(\pi_1(X),G)$ 
restricts to isomorphisms $\VV^i_r(\pi,\iota) \cong \VV^i_r(X,\iota)$ 
for $i\le 1$ and $r\ge 0$.
\end{corollary}

\begin{proof}  
The map $f$ is $1$-connected, and so the claim follows from Lemma \ref{lem:natjump-top}.
\end{proof}

\subsection{Finite families of epimorphisms}
\label{subsec:families}
We conclude this section with a setup that will often recur 
in the sequel. Let $\pi$ be a finitely generated group, and 
let $\{f \colon \pi\surj \pi_f\}_{f\in E}$ be a 
finite family of epimorphisms.  Let 
$\iota\colon G\to \GL(V)$ be a rational representation 
of $\C$-linear algebraic groups.  By Corollary  \ref{cor:jumpnat-top}, 
the natural inclusion
\begin{equation}
\label{eq:repincl}
\Hom(\pi,G) \supseteq \bigcup_{f\in E} f^{!} \Hom (\pi_f,G)
\end{equation}
induces for each $i\le 1$ and $r\ge 1$ an inclusion 
\begin{equation}
\label{eq:vincl}
\VV^i_r(\pi,\iota) \supseteq \bigcup_{f\in E} f^{!} \VV^i_r (\pi_f,\iota).
\end{equation}

One of our main goals for the remainder of this paper is to delineate  
several large classes of groups endowed with the required finite families 
of epimorphisms, for which the above two inclusions 
hold as equalities near $1$.

\section{A natural comparison between embedded jump loci}
\label{sect:nat}

This section is devoted to proving our main naturality result. 

\subsection{Functors of Artin rings}
\label{subsec:fartin}

Returning now to the setup from \S\S\ref{subsec:rep vars}-\ref{subsec:flat conn},
let $X$ be a pointed, $1$-finite space, and assume there is a 
$1$-finite $\dga$ such that $\Omega_{\k}(X)\simeq_1 A$ in $\CDGA$. 
Let $\pi=\pi_1(X)$, and let $G$ be a linear algebraic group, with Lie 
algebra $\g$.  We wish to compare the analytic germs $\Hom(\pi,G)_{(1)}$ and 
$\F(A,\g)_{(0)}$.  Let $R$ and $\overline{R}$ be the respective coordinate 
local algebras.  By Artin approximation, we may start by looking at the 
completions of these rings, $\widehat{R}$ and $\widehat{\overline{R}}$.  
Alternatively, we may analyze the corresponding functors of Artin rings, 
\begin{equation}
\label{eq:artin rings}
h_{\widehat{R}}(\art) = \Hom(\widehat{R},\art)  \ \text{\: and \:} \ 
h_{\widehat{\overline{R}}}(\art) = \Hom(\widehat{\overline{R}},\art), 
\end{equation}
for $\art$ a local Artin algebra, where $\Hom$ stands for morphisms of local 
algebras.  We recall that $h_{\widehat{R}}(\art) = \Hom(\pi, \exp(\g\otimes \m_{\art}))$, 
see~\cite{GM}, and $h_{\widehat{\overline{R}}}(\art) = \F(A, \g\otimes \m_{\art})$, 
see~\cite{DP-ccm}. 

Let $Z$ be a zig-zag of $1$-equivalences in 
$\ACDGA$ connecting $\Omega_{\k}(X)$ to $A$, as in \eqref{eq:defz}.  
It follows from Theorem \ref{thm:monbij} that the natural bijection $\alpha_Z$ 
defined in \eqref{eq:alphaz} yields an isomorphism 
\begin{equation}
\label{eq:hatiso}
\xymatrixcolsep{18pt}
\xymatrix{
\alpha_Z \colon \widehat{R} \ar[r]& \widehat{\overline{R}}
}.
\end{equation}

Assume now that $X$ and $A$ are $q$-finite, for some $q\ge 1$, and 
that $\Omega_{\k}(X)\simeq_q A$ in $\CDGA$.  As usual, let 
$\iota\colon G\to \GL(V)$ be a rational representation, with 
tangential representation $\theta\colon\g\to \gl(V)$. For each $i\le q$ 
and $r\ge 0$, we denote by $I^i_r \subseteq R$ the radical of the defining ideal 
of the germ $\VV^i_r(X,\iota)_{(1)}$ inside $\Hom(\pi,G)_{(1)}$. Similarly, we will let 
$\overline{I}^i_r \subseteq \overline{R}$ stand for the radical of the defining 
ideal of the germ $\RR^i_r(A,\theta)_{(0)}$ inside $\F(A,\g)_{(0)}$.

\begin{lemma}
\label{lem:embhat}
For any zig-zag $Z$ as above, the isomorphism 
$\alpha_Z \colon \widehat{R} \to \widehat{\overline{R}}$ from 
\eqref{eq:hatiso} identifies $\widehat{I^i_r}$ with 
$\widehat{\overline{I}}{}^i_r$, for all $i\le q$ 
and $r\ge 0$.
\end{lemma}

\begin{proof}
In view of Corollary \ref{cor:alphaspec}, we may replace the zig-zag $Z$ 
by a special zig-zag $S$.  The claim for $S$ follows from \cite[Lem.~9.9]{DP-ccm}. 
\end{proof}

Now let $f\colon X\to X'$ be a pointed map, let $\Phi\colon A'\to A$ 
be a $\dga$ map, and assume both spaces and $\dga$s are $q$-finite, 
for some $q\ge 1$.  Let $\phi\colon R\to R'$ be the morphism of local rings 
induced by the map $f_{\sharp}^{!}\colon \Hom (\pi',G)\to \Hom (\pi,G)$, 
and let  $\bar\phi\colon \overline{R}\to \overline{R}'$ be the morphism of local rings 
induced by the map $\Phi\otimes\id \colon\F(A',\g)\to \F(A,\g)$.  

Suppose there is a $q$-equivalence in $\ACDGA_0$ between the maps 
$\Omega_\k(f) \colon \Omega_\k(X')\to \Omega_\k(X)$ and 
$\Phi\colon A'\to A$.   We then obtain zig-zags $Z'$ from $\Omega_\k(X')$ 
to $A'$ and $Z$ from $\Omega_\k(X)$ to $A$; let 
$\alpha' \colon \widehat{R}' \to \widehat{\overline{R'}}$ 
and $\alpha \colon \widehat{R} \to \widehat{\overline{R}}$ be the 
corresponding isomorphisms, given by \eqref{eq:hatiso}.

\begin{lemma}
\label{lem:comm}
With the above setup, we have that $\alpha' \circ \hat\phi = \hat{\bar\phi} \circ\alpha$.
\end{lemma}

\begin{proof}
In terms of functors of Artin rings, we have that $h_{\alpha}=\mon \circ \beta_Z$ and 
$h_{\alpha'}=\mon \circ \beta_{Z'}$.  
First we show that the following diagram commutes, for every local Artin algebra $\art$. 
\begin{equation}
\label{eq:comm1}
\begin{gathered}
\xymatrix{
\F(A, \g\otimes \m_{\art}) \ar@{=}[r]& 
\DS{\frac{\F(\widetilde{A} \otimes \g\otimes \m_{\art})}%
{\G(\widetilde{A} \otimes \g\otimes \m_{\art})}} \ar^(.46){\beta_Z}[r]
& 
\DS{\frac{\F(\widetilde{\Omega}(X) \otimes \g\otimes \m_{\art})}%
{\G(\widetilde{\Omega}(X) \otimes \g\otimes \m_{\art})}}
\\
\F(A', \g\otimes \m_{\art}) 
\ar@{=}[r] \ar_{\Phi\otimes \id}[u]
& 
\DS{\frac{\F(\widetilde{A'} \otimes \g\otimes \m_{\art})}%
{\G(\widetilde{A'} \otimes \g\otimes \m_{\art})}} 
\ar^(.46){\beta_{Z'}}[r]  \ar_{\Def_{\widetilde{\Phi}\otimes \id}(\art)}[u]
& 
\DS{\frac{\F(\widetilde{\Omega}(X') \otimes \g\otimes \m_{\art})}%
{\G(\widetilde{\Omega}(X') \otimes \g\otimes \m_{\art})}}
\ar_{\Def_{\widetilde{\Omega}(f)\otimes \id}(\art)}[u]
}
\end{gathered}
\end{equation}
Plainly, it is enough to verify the commutativity of this diagram 
for an elementary $q$-equivalence in $\ACDGA_0$. In this case, 
$\beta_Z=\Def_{\widetilde{\psi}\otimes \id}(\art)$ and 
$\beta_{Z'}=\Def_{\widetilde{\psi'}\otimes \id}(\art)$, 
by construction.  The claim now follows from Lemma \ref{lem:htpy} 
and Corollary \ref{cor:man}.

Next, we show that the diagram 
\begin{equation}
\label{eq:comm2}
\begin{gathered}
\xymatrixcolsep{30pt}
\xymatrix{
\DS{\frac{\F(\widetilde{\Omega}(X) \otimes \g\otimes \m_{\art})}%
{\G(\widetilde{\Omega}(X) \otimes \g\otimes \m_{\art})}} 
\ar^(.45){\mon}[r]&
\Hom (\pi_1(X),\exp(\g\otimes \m_{\art}))
\\
\DS{\frac{\F(\widetilde{\Omega}(X') \otimes \g\otimes \m_{\art})}%
{\G(\widetilde{\Omega}(X') \otimes \g\otimes \m_{\art})}}
\ar_{\Def_{\widetilde{\Omega}(f)\otimes \id}(\art)}[u]
\ar^(.45){\mon}[r]&
\Hom (\pi_1(X'),\exp(\g\otimes \m_{\art}))
\ar_{f_{\sharp}^{!}}[u]
}
\end{gathered}
\end{equation}
commutes, where the horizontal arrows are as in Theorem \ref{thm:monbij}. 
In fact, commutativity holds even before taking quotients by the 
gauge actions, due to the naturality properties of the monodromy 
construction, as detailed in \cite[\S 6.3]{DP-ccm}. 

It is now straightforward to check that the natural transformation between 
functors of Artin rings induced by $\hat{\bar\phi}$ takes the value 
$\Phi\otimes \id$ on $\art$, whereas for $\hat\phi$ we obtain the value $f_{\sharp}^{!}$.  
The commutativity of the above two diagrams now verifies the claim. 
\end{proof}

\subsection{A natural comparison between embedded jump loci}
\label{subsec:compare}

We now consider a family of maps between pointed spaces, 
$\{f\colon X\to X_f\}_{f\in E}$, indexed by a finite set $E$, 
and we let $\{ f_{\sharp}\colon \pi \to \pi_f\}$ be the family of induced 
homomorphisms on fundamental groups. 
We also consider a family of $\ACDGA$ maps, 
$\{\Phi_f\colon A_f\to A\}_{f\in E}$, indexed by the same set,  
and we will assume that $A$ and $A_f$ are connected $\dga$s.  
Fix an integer $q\ge 1$. 

Suppose that $\Omega(f)\simeq_q \Phi_f$ in $\ACDGA_0$. 
We then have a commuting diagram as in \eqref{eq:ziggy}, 
\begin{equation}
\label{eq:ziggyf}
\begin{gathered}
\xymatrix{
Z_f: \hspace{-20pt} & \Omega(X)  & A_1 \ar_(.45){\psi_0}[l]  \ar^{\psi_1}[r] & \cdots 
& A_{\ell-1}   \ar[l]\ar^{\psi_{\ell-1}}[r] & A \, \phantom{.}
\\
Z'_f: \hspace{-20pt} & \Omega(X_f) \ar^{\Omega(f)}[u] & A'_1  
\ar^{\Phi_1}[u] \ar_(.45){\psi'_0}[l]  \ar^{\psi'_1}[r] & \cdots 
& A'_{\ell-1}   \ar_{\Phi_{\ell-1}}[u] \ar[l]\ar^{\psi'_{\ell-1}}[r] & A_f \, . \ar_{\Phi_f}[u]
}
\end{gathered}
\end{equation}
Let  $\beta_{Z_f}$ and $\beta_{Z'_f}$ be the associated natural bijections, 
defined as in display \eqref{eq:betaz}. 

\begin{definition}
\label{def:unifq}
We will say that $\Omega(f)\simeq_q \Phi_f$ in $\ACDGA_0$, 
{\em uniformly}\/ with respect to $f\in E$ 
if the bijection $\beta_{Z_f}$ is independent of $f$. 
\end{definition}

We are now in a position to state and prove our main naturality result. 
As usual, $G$ is a $\k$-linear algebraic group (where $\k=\R$ or $\C$), and 
$\g$ is its Lie algebra.  Furthermore, we consider a rational representation 
$\iota\colon G\to \GL(V)$ over $\k$, and we let $\theta\colon \g\to \gl(V)$ be the 
tangential representation.  

\begin{theorem}
\label{thm:main}
Suppose the following conditions hold:
\begin{enumerate}
\item \label{h1}  
All the above spaces and $\dga$s are $q$-finite.
\item \label{h2}
Both $f$ and $\Phi_f$ are $(q-1)$-connected maps, for all $f\in E$. 
\item \label{h3}
$\Omega(f)\simeq_q \Phi_f$ in $\ACDGA_0$, uniformly with respect to $f\in E$.
\end{enumerate}

Then we may find  local analytic isomorphisms 
$a\colon \F(A,\g)_{(0)} \isom \Hom(\pi,G)_{(1)}$ 
and $a_f\colon \F(A_f,\g)_{(0)} \isom \Hom(\pi_f,G)_{(1)}$ for all $f\in E$ 
such that the following diagram commutes, for all $f\in E$,
\[
\xymatrix{
\F(A,\g)_{(0)} \ar[r]^(.45){a}& \Hom(\pi,G)_{(1)}\, \phantom{.}\\
\F(A_f,\g)_{(0)} \ar[u]^{\Phi_f\otimes \id} \ar[r]^(.45){a_f}
&  \Hom(\pi_f,G)_{(1)}\ar[u]_{f_{\sharp}^{!}}\, .
}
\]

Moreover, for all $f\in E$, $i\le q$, and $r\ge 0$, this construction induces a 
commuting diagram of (local, reduced) embedded jump loci,
\[
\xymatrix{
(\F(A,\g), \RR^i_r(A,\theta))_{(0)} \ar[r]^(.48){a}& (\Hom(\pi,G), \VV^i_r(X,\iota))_{(1)}
\, \phantom{,}\\
(\F(A_f,\g), \RR^i_r(A_f,\theta))_{(0)}\ar[u]^{\Phi_f\otimes \id} \ar[r]^(.48){a_f}
&  (\Hom(\pi_f,G), \VV^i_r(X_f,\iota))_{(1)}\, , \ar[u]_{f_{\sharp}^{!}}
}
\]
where both horizontal arrows are isomorphisms of analytic pairs. 
\end{theorem}

\begin{proof}
Due to our connectivity assumptions, Corollaries \ref{cor:jumpnat-top} and  
\ref{cor:jumpnat-inf} apply, thereby showing that both $f_{\sharp}^{!}$ and 
$\Phi_f\otimes \id$ respect the corresponding (global) jump loci.
We will deduce all other claims from Proposition \ref{prop:simartin}.

To begin with, we denote by $\phi_f \colon R\to R_f$ and 
$\bar\phi_f \colon \overline{R}\to \overline{R}_f$ the morphisms 
of analytic algebras corresponding to the local analytic maps 
$f_{\sharp}^{!} \colon \Hom(\pi_f, G)_{(1)} \to\Hom(\pi, G)_{(1)}$
and $\Phi_f\otimes \id \colon \F(A_f, \g)_{(0)} \to\F(A, \g)_{(0)}$, 
respectively.
To verify that both $\phi_f$ and $\bar\phi_f$ are 
epimorphisms, we may use a standard, equivalent property, 
namely, the injectivity of the associated natural transformation 
between $\Hom$-functors, see for instance \cite[III.4]{T}. In turn, 
this property readily follows from the injectivity on $\art$-points 
of the corresponding morphisms between affine coordinate rings, 
for an arbitrary commutative algebra $\art$. 

For representation varieties, the map on $\art$-points 
is given by $f_{\sharp}^{!} \colon \Hom(\pi_f, G(\art)) \to\Hom(\pi, G(\art))$.  
Clearly, this map is injective, since by assumption, $f$ is $0$-connected, i.e., 
$f_{\sharp}$ is surjective. Likewise, for varieties of flat connections, the 
map on $\art$-points is given by 
$\Phi_f\otimes \id \colon \F(A_f\otimes \g\otimes \art) 
\to\F(A\otimes \g\otimes \art)$.  Again, this map is injective, 
since by assumption $\Phi_f$ is $0$-connected, i.e., injective.
This shows that the first preliminary hypotheses from 
Proposition \ref{prop:simartin} are satisfied. 

For $i\le q$ and $r\ge 0$, let $I^i_r \subseteq R$ and 
$\overline{I}^i_r \subseteq \overline{R}$ be 
defining radical ideals for the reduced analytic germs 
$\VV^i_r(X,\iota)_{(1)}$ and $\RR^i_r(A,\theta)_{(0)}$,
as in Lemma \ref{lem:embhat}.  Similarly, for $f\in E$, let 
$I^i_r(f) \subseteq R_f$ and 
$\overline{I}^i_r(f) \subseteq \overline{R}_f$ be 
defining radical ideals for 
$\VV^i_r(X_f,\iota)_{(1)}$, and $\RR^i_r(A_f,\theta)_{(0)}$.  
We deduce from display \eqref{eq:simvoid} that $I^i_r \ne R$ 
if and only if $\overline{I}^i_r \ne  \overline{R}$, which  
happens precisely when $r\le b_i\cdot \dim (V)$, where 
recall $b_i=b_i(X)=b_i(A)$.  Similarly, $I^i_r(f)$ is a proper 
ideal if and only $\overline{I}^i_r(f)$ is a proper ideal.  

Note that $I^i_r=R$ is equivalent to $\VV^i_r(X,\iota)_{(1)}=\emptyset$,  
and similarly for $I^i_r(f)$.   
By Corollary \ref{cor:jumpnat-top}, if $\VV^i_r(X,\iota)_{(1)}$ is empty, 
then $\VV^i_r(X_f,\iota)_{(1)}$ is also empty. Consequently, if 
$I^i_r$ is non-proper, then $I^i_r(f)$ is also non-proper.  
Likewise, Corollary \ref{cor:jumpnat-inf} implies the following: 
if $\overline{I}^i_r$ is non-proper, then $\overline{I}^i_r(f)$ 
is also non-proper.  

The pairs $(i,r)$ with $0\le i \le q$ and $0\le r\le  b_i\cdot \dim (V)$ 
form a finite set, which we will denote by $F$.  
Plainly, we need to verify the second claim of the theorem only for the pairs 
$(i,r)\in F$ and the maps $f\in E$ for which the ideal $I^i_r(f)$ is proper.  

By assumption \eqref{h3}, $\Omega(f)\simeq_q \Phi_f$ in $\ACDGA_0$, 
uniformly with respect to $f\in E$.  
In particular, we have a zig-zag $Z_f$ of $q$-equivalences from $\Omega(X)$ 
to $A$ and a zig-zag $Z'_f$ from $\Omega(X_f)$ to $A_f$ for each $f\in E$, as  
in diagram \eqref{eq:ziggyf}.  
Let 
\begin{equation}
\label{eq:alphaf}
\alpha_f:=\mon\circ \beta_{Z'_f}\colon \widehat{R}_f \isom \widehat{\overline{R}}_f
\end{equation}
be the isomorphism from \eqref{eq:hatiso}. By our uniformity assumption, 
the isomorphisms $\mon\circ \beta_{Z_f}$ coincide with a fixed isomorphism, 
$\alpha\colon \widehat{R} \isom \widehat{\overline{R}}$. 

It follows from Lemma \ref{lem:embhat} that the isomorphism $\alpha_f$ 
identifies the ideal $\widehat{I}{}^i_r(f)$ with $\widehat{\overline{I}}{}^i_r(f)$, 
for all $i\le q$ and $r\ge 0$, and for all $f\in E$. Likewise, the isomorphism $\alpha$ 
identifies the ideal $\widehat{I}{}^i_r$ with $\widehat{\overline{I}}{}^i_r$, 
for all $i\le q$ and $r\ge 0$. 
Finally, assumption \eqref{p4} from Proposition \ref{prop:simartin} 
follows from Lemma \ref{lem:comm}.

The desired conclusions follow from Proposition \ref{prop:simartin}, applied to the 
above  ideals.
\end{proof}

\subsection{Naturality with respect to a single map}
\label{subsec:single}

For a one-element family $E=\{f\}$, the uniform equivalence property 
from Definition \ref{def:unifq} reduces to $\Omega(f)\simeq_q \Phi_f$ in $\ACDGA_0$.
We thus have the following immediate corollary to Theorem \ref{thm:main}.

\begin{corollary}
\label{cor:unione}
Let $f\colon X\to X'$ be a continuous, $(q-1)$-connected map 
between $q$-finite, pointed spaces, for some $q\ge 1$.  Suppose $\Phi\colon A'\to A$ 
is a $(q-1)$-connected $\dga$ map between $q$-finite $\dga$s such that 
$\Omega(f)\simeq_q \Phi$ in $\ACDGA_0$.  
We then may find  local analytic isomorphisms $a$ and $a'$ 
which fit into the diagram
\[
\xymatrix{
\F(A,\g)_{(0)} \ar[r]^(.4){a}& \Hom(\pi_1(X),G)_{(1)}\, \phantom{.}\\
\F(A',\g)_{(0)} \ar[u]^{\Phi\otimes \id} \ar[r]^(.4){a'}
&  \Hom(\pi_1(X'),G)_{(1)}\ar[u]_{f_{\sharp}^{!}}\, .
}
\]
Furthermore, for all $i\le q$, and $r\ge 0$, this construction induces a 
commuting diagram of (local, reduced) embedded jump loci,
\[
\xymatrix{
(\F(A,\g), \RR^i_r(A,\theta))_{(0)} \ar[r]^(.43){a}& (\Hom(\pi_1(X),G), \VV^i_r(X,\iota))_{(1)}
\, \phantom{,}\\
(\F(A',\g), \RR^i_r(A',\theta))_{(0)}\ar[u]^{\Phi\otimes \id} \ar[r]^(.43){a'}
&  (\Hom(\pi_1(X'),G), \VV^i_r(X',\iota))_{(1)}\, , \ar[u]_{f_{\sharp}^{!}}
}
\]
where both horizontal arrows are isomorphisms of analytic pairs. 
\end{corollary}

Here is a situation where this type of property holds. 

\begin{lemma}
\label{lemma:qeqmodel}
Let $f\colon X\to X'$ be a continuous map between pointed spaces, and 
assume $H^{\hdot}(f)$ is  $q$-connected, for some $q\ge 1$.  Let $A$ be a connected 
$\dga$, and suppose $A$ is a $q$-model for $X$. Then 
$\Omega(f)\simeq_q \id_A$ in $\ACDGA_0$. 
\end{lemma}

\begin{proof}
Consider the following commuting diagram in $\ACDGA_0$,
\begin{equation}
\label{eq:commute triangle}
\begin{gathered}
\xymatrixrowsep{6pt}
\xymatrix{ 
\Omega(X) \\
& \mathcal{M}_q \ar_(.42){\rho}[ul]   \ar^(.42){\rho'}[dl]   \ar^{\bar\rho'}[r] & A\, ,\\
\Omega(X') \ar^{\Omega(f)}[uu]
}
\end{gathered}
\end{equation}
where $\rho'$ and $\bar\rho'$ are $q$-minimal model maps 
provided by the assumption that $\Omega(X')\simeq_q \Omega(X)\simeq_q A$. 
Clearly, the map $\rho=\Omega(f) \circ \rho'$ is a 
$q$-equivalence, since both $\Omega(f)$ and $\rho'$ are.  
Hence, $\Omega(f)\simeq_q \id_A$ in $\ACDGA_0$,  
and the claim follows.
\end{proof}

\begin{corollary}
\label{cor:qeqmodel}
Fix $q\ge 1$. 
Let $f\colon X\to X'$ be a $(q-1)$-connected map between $q$-finite, pointed 
spaces, such that $H^{\hdot}(f)$ is $q$-connected. Let $A$ be a $q$-finite $\dga$, and 
suppose $A$ is a $q$-model for $X$. Then the conclusions of Corollary \ref{cor:unione} 
hold for $A'=A$ and $\Phi=\id_A$. 
\end{corollary}

We conclude with one more class of spaces and maps where  
Corollary \ref{cor:unione} applies. 
 
\begin{prop}
\label{prop:fmap}
Let $f\colon X\to X'$ be a $(q-1)$-connected map between $q$-finite, 
pointed spaces, for some $q\ge 1$. Assume that $f$ is formal over $\k$.   
Then the conclusions of Corollary \ref{cor:unione} hold for 
$\Phi=H^{\hdot}(f)\colon H^{\hdot}(X',\k)\to H^{\hdot}(X,\k)$. 
\end{prop}

\begin{proof}
Our connectivity hypothesis implies that $H^1(f)$ is injective.  
The claim follows at once from Proposition \ref{prop:mainf}.
\end{proof}

\section{K\"{a}hler manifolds}
\label{sect:kahler}

In this section we show that pointed holomorphic maps between compact \Ka 
manifolds can be {\em uniformly}\/ modeled by the homomorphisms induced 
in (real) cohomology.  As an application, we derive a structural result on 
the germs at the origin of rank $2$ embedded jump loci of \Ka groups. 

\subsection{Essentially rank one flat connections}
\label{subsec:prelim}

We start with some preliminary lemmas. 

\begin{lemma}
\label{lem:l1k}
A non-abelian Lie subalgebra $\g\subseteq \sl_2(\C)$ is either equal to 
$\sl_2(\C)$ or is isomorphic to the standard Borel subalgebra, $\sol_2(\C)$.
\end{lemma}

\begin{proof}
Easy exercise.
\end{proof}

We now recall a few facts from \cite{MPPS}.  Let $A$ be a 
$\dga$, let $\g$ be a Lie algebra, and let 
$\F(A,\g)\subset A^1\otimes \g$ be the set of $\g$-valued flat 
connections on $A$.  Let us define $\F^1(A,\g)$ to be the subset 
of $A^1\otimes \g$  consisting of all tensors of the form 
$\eta \otimes g$ with $d \eta=0$.  
We also fix a finite-dimensional representation $\theta\colon \g\to \gl (V)$, 
and define $\Pi(A,\theta)$ to be the subset of $\F^1(A,\g)$ 
consisting of all tensors as above which also satisfy $\det(\theta(g))=0$. 
When $A$ is $1$-finite and $\g$ is finite-dimensional, both $\F^1(A,\g)$ and $\Pi(A,\theta)$ 
are closed, homogeneous subvarieties of $\F(A,\g)$.  Moreover, if $H^i(A)\ne 0$, then 
$\Pi(A,\theta)\subseteq \RR^i_1(A,\theta)$.

\begin{lemma}
\label{lem:l2k}
Under the above finiteness assumptions,
every $\dga$ map $\Phi\colon A'\to A$ induces algebraic maps, $\Phi\otimes \id 
\colon \F^1(A',\g) \to \F^1(A,\g)$ and $\Phi\otimes\id \colon  \Pi(A',\theta) \to \Pi(A,\theta)$.  Moreover, 
if $H^1(\Phi)$ is an isomorphism, then both these algebraic maps are isomorphisms. 
\end{lemma}

\begin{proof}
Follows directly from the definitions.
\end{proof}

\begin{lemma}
\label{lem:l4k}
Let $\Phi\colon A' \to A$ be a $\dga$ map, where 
$A' = (\bwedgedot U, d=0)$ with $0< \dim U<\infty$ and $A$ is $1$-finite, 
and assume $H^1(\Phi)$ is an isomorphism.  Also let $\g  \subseteq \sl_2(\C)$ 
be a Lie subalgebra, and let $\theta\colon \g\to \gl (V)$ be a finite-dimensional 
representation. Then the following hold:
\begin{enumerate}
\item \label{l41}
$\Phi\otimes \id$ induces an isomorphism between  $\F(A',\g)$ and $\F^1(A,\g)$. 
\item \label{l42}
$\Phi\otimes \id$ induces an isomorphism between  $\RR^1_1(A',\theta)$ and $\Pi(A,\theta)$. 
\end{enumerate} 
\end{lemma}

\begin{proof}
By construction, $A'$ is the Chevalley--Eilenberg cochain $\dga$ of the abelian 
Lie algebra $U$. Hence, by \cite[Lem.~4.14]{MPPS}, we have that $\F(A',\g)=\F^1(A',\g)$. 
The first claim follows at once from Lemma \ref{lem:l2k}.

It is easily checked that $\RR^1_1(A')=\{0\}$.  Since $\F(A',\g)=\F^1(A',\g)$, we 
may apply \cite[Cor.~3.8]{MPPS} to conclude that $\RR^1_1(A',\theta)=\Pi(A',\theta)$.
The second claim now follows  from Lemma \ref{lem:l2k}.
\end{proof}

\subsection{The uniformity property}
\label{subsec:uk}

Let $\{f\colon M\to M_f\}_{f\in E}$ be a finite family of pointed, holomorphic maps 
between compact \Ka manifolds.  Each map $f\in E$ induces a homomorphism 
$H^{\hdot}(f)\colon H^{\hdot}(M_f)\to H^{\hdot}(M)$ between the respective 
cohomology algebras with coefficients in $\k=\R$ or $\C$.  In fact, these 
homomorphisms may be viewed as $\ACDGA_0$ maps, by setting the 
differentials to be zero, and taking the augmentations given by the basepoints. 

\begin{prop}
\label{prop:unifk}
In the above setup, $\Omega(f)\simeq H^{\hdot}(f)$ in $\ACDGA_0$, 
uniformly with respect to $f\in E$. 
\end{prop}

\begin{proof}
To prove the claim, it is enough to show there is a {\em functorial}\/ 
zig-zag of quasi-isomorphisms in $\ACDGA$ connecting $\Omega(M)$ 
to $H^{\hdot}(M)$, for any pointed, compact \Ka manifold $M$.  In order 
to construct such a zig-zag, we proceed in two steps, following 
\cite[\S 6]{DGMS} and \cite[\S 11]{FHT}.  It is enough to work 
over $\k=\R$.  

Let $(\Omega_{\DR} (M),d)$ be the de Rham $\dga$ of the underlying 
differentiable manifold, with $d$ the exterior differential.  Set $d^c=J^{-1} d J$, 
where $J$ is the complex structure on the tangent bundle to $M$. This 
gives another $\dga$, $(\Omega_{\DR} (M), d^c)$.  By the first proof 
of the Main Theorem from \cite{DGMS}, there is a zig-zag of 
quasi-isomorphisms in $\CDGA$ connecting $(\Omega_{\DR}(M),d)$ to 
$(H^{\hdot}(\Omega_{\DR} (M), d^c),d=0)$, natural with respect to 
holomorphic maps.  Taking homomorphisms induced in cohomology,  
we obtain a natural zig-zag of $\dga$ quasi-isomorphisms connecting 
$(H^{\hdot}_{\DR}(M),d=0)$ to $(H^{\hdot}(\Omega_{\DR} (M), d^c), d=0)$.  
In fact, both zig-zags are in $\ACDGA$, since all their terms 
are equal to $\R$ when $M$ is a point.  Combining these two 
zig-zags, we obtain a functorial zig-zag of quasi-isomorphisms 
in $\ACDGA$ from $\Omega_{\DR}(M)$ to $(H^{\hdot}_{\DR}(M), d=0)$.

On the other hand, the proof of the de Rham theorem from \cite{FHT} 
provides a zig-zag of quasi-isomor\-phisms in $\CDGA$ connecting  
$\Omega_{\DR}(M)$ to $\Omega_{\R}(M)$, which is natural with 
respect to differentiable maps.  An argument as above shows that this zig-zag is in 
$\ACDGA$.   Putting things together, and using the classical de Rham 
theorem, we arrive at the desired conclusion. 
\end{proof}

\subsection{Admissible maps and rank $1$ jump loci}
\label{subsec:arapura}

A connected, complex manifold $M$ is said to be a {\em quasi-compact \Ka 
manifold} if there is compact \Ka manifold  $\overline{M}$ 
and a normal crossing divisor $D\subset M$ such that 
$M=\overline{M}\setminus D$. Of course, all compact \Ka 
manifolds belong to this class.  Furthermore, if $M$ is 
an irreducible smooth, complex quasi-projective variety, 
or, for short, a {\em quasi-projective manifold}, then $M$ is 
also of this type, by resolution of singularities.

Given a quasi-compact \Ka manifold $M$, there is a certain finite 
family of pointed holomorphic maps, $\{f\colon M \to M_f\}_{f\in \cE(M)}$, 
with each $M_f$ a quasi-projective manifold, which is intimately related to 
the structure near $1$ of the characteristic variety $\VV^1_1(M)$.  

More precisely, a holomorphic map onto a smooth complex curve, $f\colon M \to M_f$,
is said to be {\em admissible}\/ if it extends to a holomorphic surjection with connected
fibers, $\overline{f}\colon \overline{M} \to \overline{M}_f$, where $\overline{M}$
(respectively $\overline{M}_f$) is a \Ka compactification of $M$ (respectively $M_f$)
obtained by adding a normal crossing divisor.
It is known that, up to reparametrization at the target, there is a finite family 
$\cE(M)$ of such maps with the property that $\chi(M_f)<0$. 
For each $f\in \cE(M)$, let us write $\pi=\pi_1(M)$ and $\pi_f=\pi_1(M_f)$. 
It is readily seen that the induced homomorphism 
on fundamental groups, $f_{\sharp}\colon \pi\to \pi_f$, is surjective. 
Work of Arapura \cite{Ar} shows that  the correspondence 
\begin{equation}
\label{eq:arapura}
f \leadsto f_{\sharp}^{!} \Hom (\pi_f, \C^{\times})
\end{equation}
establishes a bijection between the set $\cE(M)$ and the 
set of positive-dimensional, irreducible components of the 
characteristic variety $\VV^1_1(M)$ passing through $1$.  

For a pointed CW-space $M$ with fundamental group $\pi$, we denote by
$f_0 \colon M \to K(\pi_{\abf}, 1)$ the classifying map determined up to homotopy 
by the property that $(f_0)_{\sharp}= \abf \colon \pi \surj \pi_{\abf}$, 
where $\abf$ is the canonical projection of the group $\pi$ onto its 
maximal torsion-free abelian quotient. When $M$ is a 
quasi-compact \Ka manifold, we set 
\begin{equation}
\label{eq:em}
E(M)=\cE(M) \cup \{f_0\}. 
\end{equation}

In the rank one case, i.e., when $\iota=\id_{\C^{\times}}$, both inclusions,
\eqref{eq:repincl} and \eqref{eq:vincl} for $i=r=1$, become equalities 
near the origin $1$, for the family $\{ f_{\sharp} \mid f\in E(M) \}$.

\begin{theorem}
\label{thm:qk germs}
Let $M$ be a quasi-compact \Ka manifold, and let $\pi=\pi_1(M)$. Then,
\begin{align}
\label{eq:repincl-rk1}
\Hom(\pi,\C^{\times})_{(1)} &=
\bigcup_{f\in E(M)} f_{\sharp}^{!} \Hom (\pi_f,\C^{\times})_{(1)}, 
\\
\label{eq:vincl-rk1}
\VV^1_1(\pi)_{(1)} &= \bigcup_{f\in E(M)} f_{\sharp}^{!} \VV^1_1 (\pi_f)_{(1)}.
\end{align}
\end{theorem}

\begin{proof}
The first claim is easily verified. Indeed, the abelianization map, 
$\pi \surj \pi_{\ab}$, induces an isomorphism of character groups, 
while the map induced by the natural projection 
$\pi_{\ab} \surj \pi_{\abf}$ identifies $\Hom(\pi_{\abf}, \C^{\times})$ 
with the identity component of $\Hom(\pi_{\ab}, \C^{\times})$.  It follows that 
$(f_0)_{\sharp}=\pi_{\abf}$ induces an isomorphism 
between germs at $1$ of $\C^{\times}$-representation varieties, and so
\eqref{eq:repincl-rk1} holds.

The second claim is much more subtle. Since $\chi(M_f)<0$, 
it is easily seen that $\VV^1_1(\pi_f)=\VV^1_1(M_f)=\Hom (\pi_f, \C^{\times})$, 
for $f\in \cE(M)$. If $b_1(M)=0$, we know from \eqref{eq:simvoid} 
that $\VV^1_1(\pi)_{(1)}=\emptyset$, and we are done. If $b_1(M)>0$, then
either $\VV^1_1(\pi)_{(1)}=\{1\}$,  or all irreducible components of $\VV^1_1(\pi)$
passing through $1$ are positive-dimensional. In the first case we are done, since
$1\in (f_0)_{\sharp}^{!} \VV^1_1(\pi_{\abf})$. In the second case, equality
\eqref{eq:vincl-rk1} follows from the aforementioned deep results of Arapura.
This completes the proof.
\end{proof}

\subsection{Rank $2$ embedded jump loci of K\"{a}hler manifolds}
\label{subsec:pfrk2k}

Let $M$ be a compact \Ka manifold with fundamental group $\pi$. 
A map $f\colon M \to M_f$ is admissible in the sense from 
\S\ref{subsec:arapura} if $M_f$ is a compact Riemann surface 
and $f$ is a holomorphic surjection with connected fibers. The
Albanese map, $f_0\colon M \to \Alb(M)$, is a holomorphic map 
between compact \Ka manifolds which classifies the canonical 
projection, $\pi \surj \pi_{\abf}$.

Our next goal is to extend the rank $1$ results 
\eqref{eq:repincl-rk1}--\eqref{eq:vincl-rk1}
to the rank $2$ case.  We start with a lemma.

\begin{lemma}
\label{lem:l3k}
Let $M$ be a compact \Ka manifold with $b_1(M)>0$.  Let $\g$ 
be a non-abelian Lie subalgebra of $\sl_2(\C)$, and let 
$\theta\colon \g\to \gl (V)$ be a finite-dimensional representation. 
Then the following equalities hold:
\begin{align}
\label{eq:flateq}
\F(H^{\hdot}(M),\g) &= \F^1(H^{\hdot}(M),\g)\cup 
\bigcup_{f\in \mathcal{E}(M)}  f^{!} (\F(H^{\hdot}(M_f),\g)),
\\
\label{eq:reseq}
\RR^1_1(H^{\hdot}(M),\theta) &= \Pi(H^{\hdot}(M),\theta)\cup 
\bigcup_{f\in \mathcal{E}(M)}  f^{!} (\F(H^{\hdot}(M_f),\g)),
\end{align}
where all $\dga$s are endowed with zero differential. 
\end{lemma}

\begin{proof}
This is proved in \cite[Cor.~7.2]{MPPS} for a $1$-formal, quasi-projective 
manifold $M$. By results from \cite{DPS-duke}, the proof also works for 
$1$-formal, quasi-compact \Ka manifolds, in particular, for a compact \Ka 
manifold $M$. 
\end{proof}

\begin{theorem}
\label{thm:rk2k}
Let $M$ be a compact \Ka manifold with fundamental group $\pi$, and set 
$E(M)=\cE(M) \cup \{f_0\}$ as in \eqref{eq:em}. Let $G$ be $\C$-linear algebraic 
group with non-abelian Lie algebra $\g\subseteq \sl_2(\C)$, 
and let $\iota\colon G\to \GL(V)$ be a rational representation.  Then, 
\begin{align}
\label{eq:repincl-k}
\Hom(\pi,G)_{(1)} &= \bigcup_{f\in E(M)} f_{\sharp}^{!} \Hom (\pi_f,G)_{(1)}, 
\\
\intertext{and, for $i=r=1$ or $i=0$ and $r\ge 1$,}
\label{eq:vincl-k}
\VV^i_r(\pi,\iota)_{(1)}&= \bigcup_{f\in E(M)} f_{\sharp}^{!} \VV^i_r (\pi_f,\iota)_{(1)}.
\end{align}
\end{theorem}

\begin{proof}
We wish to apply Theorem \ref{thm:main} with $q=1$ to the family 
of pointed maps $\{f\colon M \to M_f\}_{f\in E(M)}$ and $\dga$ maps 
$\Phi_f=\{H^{\hdot} (f) \colon H^{\hdot}(M_f) \to H^{\hdot}( M)\}_{f\in E(M)}$, 
where all the differentials are set equal to $0$.   For that, 
we need to verify that the three hypotheses of the theorem hold. 

First, all spaces and $\dga$s in question are $1$-finite (in fact, $\infty$-finite).  
Second, each map $f_{\sharp}$ is surjective, hence each $f\in E(M)$ is $0$-connected.  
Thus, $H^{\hdot} (f)$ is also $0$-connected. 
Finally, by Proposition \ref{prop:unifk}, $\Omega(f)\simeq H^{\hdot}(f)$ 
in $\ACDGA_0$, uniformly with respect to $f\in E(M)$. 

In the case when $i=0$, equality \eqref{eq:repincl-k} clearly implies 
equality \eqref{eq:vincl-k}, by Corollary \ref{cor:jumpnat-top}. Thus, we may assume 
$i=r=1$ in \eqref{eq:vincl-k}.

Suppose now that $b_1(M)=0$. By \eqref{eq:simvoid}, we have that 
$\VV^1_1(\pi,\iota)_{(1)}=\emptyset$.   Therefore, equality \eqref{eq:vincl-k} follows trivially.   
Moreover, the natural map $\Omega (K(1,1)) \to \Omega (K(\pi,1)) $ 
is a $1$-equivalence; hence, $\pi$ has the same $1$-minimal model as the trivial group. 
It then follows from \cite[Thm.~A]{DP-ccm} that $\Hom(\pi,G)_{(1)} = \{1\}$. Therefore, 
equality \eqref{eq:repincl-k} holds trivially. 
 
Thus, we may also assume that $b_1(M)>0$. We deduce from formula \eqref{eq:flateq} 
and Lemma \ref{lem:l4k}, part \eqref{l41} that
\begin{equation}
\label{eq:flateq-bis}
\F(H^{\hdot}(M),\g) = f_0^{!} \F(H^{\hdot}(M_0),\g)\cup 
\bigcup_{f\in \mathcal{E}(M)}  f^{!} (\F(H^{\hdot}(M_f),\g)),
\end{equation}
where $M_0$ denotes the Albanese variety $\Alb(M)\simeq K(\pi_{\abf},1)$,  
and $H^{\hdot}(M_0)=\bwedgedot H^1(M)$.  
By taking germs at the origin and using the naturality properties from 
Theorem \ref{thm:main}, formula \eqref{eq:flateq-bis} implies that equality 
\eqref{eq:repincl-k} holds.

Similarly, formula \eqref{eq:reseq} 
and Lemma \ref{lem:l4k} part \eqref{l42} together imply that
\begin{equation}
\label{eq:reseq-bis}
\RR^1_1(H^{\hdot}(M),\theta) =  f_0^{!}   \RR^1_1(H^{\hdot}(M_0),\theta)\cup 
\bigcup_{f\in \mathcal{E}(M)}  f^{!} (\F(H^{\hdot}(M_f),\g)).
\end{equation}

For each $f\in \mathcal{E}(M)$, note that $M_f$ is a $2$-dimensional 
CW-complex with $\chi(M_f)<0$. An easy Euler characteristic argument 
then shows that $\VV^1_1(\pi_f,\iota)=\VV^1_1(M_f,\iota)=\Hom(\pi_f,G)$.
Again by Theorem \ref{thm:main}, formula \eqref{eq:reseq-bis} now implies 
that equality \eqref{eq:vincl-k} holds.  This completes the proof.
\end{proof}

\begin{remark}
\label{rem:tri}
In \cite[Cor.~B]{LPT}, Loray, Pereira, and Touzet  prove the following 
result, which refines earlier results of Corlette and Simpson \cite{CS}. 
Let $X$ be a quasi-projective 
manifold, and let $\rho\in \Hom(\pi_1(X),\SL_2(\C))$ be a representation which is not 
virtually abelian.  Then there is an orbifold morphism, $f\colon X\to Y$, such that 
the associated representation, $\tilde\rho\in \Hom( \pi_1(X), \PSL_2(\C))$, belongs 
to $f_{\sharp}^{!} \Hom(\pi_1^{\orb}(Y),  \PSL_2(\C))$, where $Y$ is either a 
$1$-dimensional complex orbifold, or a polydisk Shimura modular orbifold. 

For a finitely generated group $\pi$ and a linear algebraic group $G$, 
the abelian part of the representation variety $\Hom(\pi,G)$ coincides 
near $1$ with $\abf^{!} \Hom(\pi_{\abf},G)_{(1)}$.  Indeed, \cite[Thm.~A]{DP-ccm} 
implies that the canonical projection $\pi_{\ab}\surj \pi_{\abf}$ induces an 
isomorphism of germs at the origin of the respective representation varieties. 

This remark shows that formula \eqref{eq:repincl-k} from Theorem \ref{thm:rk2k} 
may be viewed as a compact \Ka analogue near $1$ of \cite[Cor.~B]{LPT}.  
In this context, it provides a simpler classification: the representation 
$\rho$ is either abelian, or it pulls back via an admissible map 
from a compact Riemann surface of genus $g>1$. 
\end{remark}

\subsection{The main difficulty in the non-abelian case}
\label{subsec:2difficult}

The naturality property from \cite[Thm. B(2)]{DP-ccm} is a consequence of the 
following fact, which holds in the abelian case. Let $X$ be a $1$-finite space 
and $A^{\hdot}$ a $1$-finite $\dga$.  The fact that $\Omega_{\k}(X)\simeq_1 A$ 
in $\CDGA$ means that there is a zig-zag of $1$-equivalences in $\CDGA$, 
\begin{equation}
\label{eq:zig omega}
\xymatrix{\Omega_{\k}(X) & \NN  \ar_(.32){\psi}[l]  \ar^{\bar\psi}[r] & A
},
\end{equation}
where $\NN=(\bwedgedot U, d)$ is a $1$-minimal $\dga$, 
see Lemma \ref{lem:zig}. 
If $\g$ is an abelian Lie algebra, it  follows from the definitions 
that $\F(B,\g)=H^1(B)\otimes \g$,  for any connected $\dga$ $B$. 
Applying this observation to the map $\bar\psi$, we conclude that, 
in the abelian case, there is a bijection 
\begin{equation}
\label{eq:abelbij}
\xymatrix{ \bar\psi\otimes \id \colon \F(\NN,\g) \ar^(.6){\simeq}[r]&  \F(A,\g) }.
\end{equation}

For a non-abelian Lie algebra $\g$, though, this map is not necessarily surjective.  
To illustrate this phenomenon, we first need a lemma.

\begin{lemma}
\label{lem:barpsi}
If $\g=\sl_2(\C)$ and $\bar\psi\otimes \id $ is surjective, 
then $\F(A,\g)=\F^1(A,\g)$.
\end{lemma}

\begin{proof}
The $\cdga$ $\NN=(\bwedgedot U, d)$ 
comes endowed with the canonical filtration, $\NN=\bigcup_{n\ge 1} \NN_n$, 
of a $1$-minimal $\dga$, where each $\dga$ $\NN_n$ 
is of the form $(\bwedgedot U_n, d)$. 
Since $\dim H^1(\NN)<\infty$, $\mathcal{N}_n$ is the 
cochain algebra of a certain finite-dimensional, nilpotent Lie algebra. 
Since $\g=\sl_2(\C)$, it follows from \cite[Lem.~4.14]{MPPS} that 
$\F(\NN_n,\g)=\F^1(\NN_n,\g)$, for each $n\ge 1$. 
Hence, $\F(\NN,\g)=\F^1(\NN,\g)$. Since 
$\bar\psi\otimes \id$ is surjective, we conclude that $\F(A,\g)=\F^1(A,\g)$.
\end{proof}

\begin{example}
\label{ex:nolift}
Let $\Sigma_g$ be a compact Riemann surface of genus $g>1$.  Since  
$\Sigma_g$ is a formal space, the $\dga$ $A=(H^{\hdot}(\Sigma_g),d=0)$ 
is a finite model for it. Let $\NN$ be a $1$-minimal model for $A$, and 
let $\bar\psi \colon \NN\to A$ be the corresponding map. It follows 
from \cite[Lem.~7.3]{MPPS} that $\F(A,\g) \ne \F^1(A,\g)$. 
By Lemma \ref{lem:barpsi}, then, the map $\bar\psi\otimes \id $ 
is not surjective.
\end{example}

Thus, in the case when $G=\SL_2(\C)$, we have no natural analytic 
map $\F(A,\g) \to \Hom(\pi_1(X),G)$.  This is the reason why we have 
to construct a local analytic isomorphism between the germs 
at the origin of the two varieties, in a manner which is compatible with 
both continuous maps and $\dga$ maps, using the simultaneous 
Artin approximation technique from Proposition \ref{prop:simartin}.

\section{Principal bundles}
\label{sect:kfree}

In this section, we apply our theory to principal bundles. 
When the base manifold is formal, we obtain a structural 
result for the germs at the origin of  rank $2$ embedded jump loci of the total space.  

\subsection{Two-element families with the uniform property}
\label{subsec:two guys}

In the applications of Theorem \ref{thm:main}, we also need to take into account 
the projection of a group $\pi$ onto its maximal torsion-free abelian quotient, 
$\abf\colon \pi\surj \pi_{\abf}$. 

\begin{theorem}
\label{prop:2unif}
Let $f\colon M\to N$ be a continuous, pointed map.  Denote by $f_0\colon M \to 
K(\pi_1(M)_{\abf},1)$ the classifying map for the above projection.  Suppose that 
$M$ and $N$ are $q$-finite, for some $q\ge 1$, and that 
$\Omega(f)\simeq_q \Phi$ in $\ACDGA_0$, 
where $\Phi\colon A_N\to A_M$ is a $\dga$ map between $q$-finite objects. Set  
$A^{\hdot}_0= (\bwedgedot H^1(M), d=0)$.  There is then a $\dga$ map 
$\Phi_0\colon A_0 \to A_M$ inducing an isomorphism on $H^1$, and such that 
$\Omega(f_0)\simeq_q \Phi_0$ in $\ACDGA_0$, uniformly with respect to the 
families $\{f,f_0\}$ and $\{\Phi,\Phi_0\}$. Moreover, if $f$ and $\Phi$ are $0$-connected 
maps, then all the hypotheses from Theorem \ref{thm:main} are satisfied for $q=1$. 
\end{theorem}

\begin{proof}
The assumption that $\Omega(f)\simeq_q \Phi$ provides a zig-zag $Z$ 
of $q$-equivalences in $\ACDGA$ connecting $\Omega_{\k}(M)$ to $A_M$. 
Let $\rho\colon \NN \to \Omega_{\k}(M)$ be a $\pi_1$-adapted $1$-minimal model map, 
as in \cite[\S6.4]{DP-ccm}.  This map can be extended to a $q$-minimal model map 
$\rho :\M \to \Omega_{\k}(M)$. By Proposition \ref{prop:ztos}, there is a special zig-zag 
$S$ of the form 
$\xymatrixcolsep{14pt}
\xymatrix{
 \Omega_{\k}(M) & \mathcal{M} \ar_(.36){\rho}[l]  \ar^{\bar\rho}[r] & A_M}$
such that $\beta_Z=\beta_S$.

Now, as explained in \cite[\S6.4]{DP-ccm}, there is a canonical $\dga$ inclusion, 
$j\colon A_0\inj \NN$, inducing an isomorphism on $H^1$. It follows that the 
map $\Phi_0=\bar{\rho} \circ j\colon A_0 \to A_M$ has the same property.  Putting 
things together, we obtain the following commuting diagram in $\ACDGA$:
\begin{equation}
\label{eq:abfdgr}
\begin{gathered}
\xymatrixcolsep{30pt}
\xymatrix{
&\M \ar_{\rho}[dl]  \ar^{\bar\rho}[dr] \\
\Omega(M) & \NN \ar@{^{(}->}^{i}[u]  \ar_(.38){\rho}[l]  \ar^(.47){\bar\rho}[r]  
& A_M\,\phantom{.} \\
\Omega(\pi_{\abf}) \ar^{\Omega(f_0)}[u] & 
A_0 \ar_(.35){\rho_{\abf}}[l] \ar@{=}[r]   \ar@{^{(}->}^{j}[u]
& A_0\, .  \ar_{\Phi_0}[u]
}
\end{gathered}
\end{equation}

Both upper-diagonal arrows are $q$-equivalences, and both lower-horizontal 
arrows are $\infty$-equivalences.  It follows that $\Omega(f_0)\simeq_q \Phi_0$ 
in $\ACDGA_0$, as claimed.  The uniform property follows from 
the equality $\beta_Z=\beta_S$. It is obvious that the map $f_0$ is $0$-connected. 
Finally, since $H^1(\Phi_0)$ is injective, the map $\Phi_0$ is also $0$-connected.
\end{proof}

\subsection{Models for principal bundle projections}
\label{subsec:princ proj}

Let $K$ be a compact, connected real Lie group acting 
freely on a closed, smooth manifold $M$. 
Let $N=M/K$ be the orbit space, and let $f\colon M\to N$ 
be the projection map of the resulting principal $K$-bundle. 
Of course, both $M$ and $N$ have the homotopy type of 
a finite CW-complex.  We will fix compatible basepoints 
for $M$ and $N$.  Note that $f$ is $0$-connected, by 
the exact homotopy sequence of the fibration 
$K\to  M \xrightarrow{f} N$ and the connectivity of $K$.  

By a classical result of H.~Hopf, the cohomology algebra of $K$ 
(with coefficients in a field $\k$ of characteristic  $0$) is of the form  
$H^{\hdot}(K)=\bigwedge P^{\hdot}$, where $P^{\hdot}$ is a finite-dimensional, 
oddly graded $\k$-vector space. Let $[\tau]\colon P^{\hdot}\to H^{\hdot+1}(N)$ 
be the transgression in the Serre spectral sequence of our fibration.

Suppose $A_N$ is a $\dga$ model for $N$, so that there is a zig-zag 
of quasi-isomorphisms connecting $\Omega_{\k}(N)$ to $A_N$.  Such 
a zig-zag yields an isomorphism of $H^{\hdot}(N)$ with $H^{\hdot}(A_N)$. 
Let  $\tau\colon P^{\hdot}\to Z^{\hdot+1}(A_N)$ be a lift of $[\tau]$.  
As noted in \S\ref{subsec:minmod}, the Hirsch extension 
$A_M:=A_N\otimes_{\tau} \bigwedge P$
is well-defined, up to a $\dga$ isomorphism extending $\id_{A_N}$. 

\begin{prop}
\label{prop:hmodel}
Let $f\colon M\to N$ be the projection map of a principal $K$-bundle as above, 
and suppose $N$ admits a finite model $A_N$. Let $\Phi\colon A_N\inj A_M$ 
be the canonical $\dga$ inclusion.  Then $A_M$ is a finite model for $M$, 
and both $f$ and $\Phi$ are $0$-connected maps. Moreover, 
$\Omega(f)\simeq \Phi$ in $\ACDGA_0$, and thus 
the conclusions of Corollary \ref{cor:unione} hold for $q=1$.
\end{prop}

\begin{proof}
Clearly, since $A_N$ is a finite $\dga$, then $A_M$ is also a finite $\dga$.  
Equally clearly, the map $\Phi$ is $0$-connected.
By the classical Hirsch Lemma (see \cite[pp.~216--218]{FHT}) there is 
a commutative diagram in $\ACDGA$,
\begin{equation}
\label{eq:hmodel}
\begin{gathered}
\xymatrixcolsep{30pt}
\xymatrix{
\Omega(M)  &\Omega(N)\otimes_{\tau} \bwedge P \ar_(.55){h}[l]
\\
&\Omega(N) \ar^{\Omega(f)}[lu]  \ar@{^{(}->}[u]
}
\end{gathered}
\end{equation}
where $h$ is an $\infty$-equivalence.   Since by assumption 
$\Omega(N)\simeq A_N$, there is a minimal $\cdga$
$\NN$, connected by quasi-isomorphisms $\psi\colon \NN\to \Omega(N)$ 
and $\bar\psi\colon \NN\to A_N$.  We then obtain a  
commuting diagram in $\CDGA$,
\begin{equation}
\label{eq:hirschdgr}
\begin{gathered}
\xymatrixcolsep{22pt}
\xymatrix{
\Omega(N)\otimes_{\tau} \!\bwedge P & 
\Omega(N) \otimes_{\tau_N} \!\bwedge P \ar_(.5){\simeq}[l] &
\NN \otimes_{\tau_{\NN}} \!\bwedge P  \ar_(.45){\psi\otimes \id}[l] 
\ar^(.5){\bar\psi\otimes \id}[r]   & 
A_N \otimes_{\tau_{A}} \!\bwedge P
 \ar^(.65){\simeq}[r] & A_M
\\
&\Omega(N) \ar@{^{(}->}[u] \ar@{_{(}->}[lu] & \NN  \ar@{^{(}->}[u] \ar_(.45){\psi}[l] 
\ar^(.5){\bar\psi}[r]  & A_N \ar@{^{(}->}[u] \ar@{^{(}->}_{\Phi}[ur]
}
\end{gathered}
\end{equation}
where the transgression $[\tau_{\NN}]$ is identified with $[\tau_N]$ 
using $H^{\hdot}(\psi)$, while $[\tau_{\NN}]$ is identified with $[\tau_A]$ 
using $H^{\hdot}(\bar\psi)$.   Note that all $\dga$s in \eqref{eq:hirschdgr} 
are augmented, and all maps respect augmentations. 
By \cite[Lem.~14.2]{FHT}, the maps $\psi\otimes \id$ and $\bar\psi\otimes \id$ 
are quasi-isomorphisms, since both $\psi$ and $\bar\psi$ are.  Splicing together  
diagrams \eqref{eq:hmodel} and \eqref{eq:hirschdgr} we reach the desired conclusions. 
\end{proof}

\subsection{Embedded jump loci of principal bundles}
\label{subsec:cjl bundle}

Before stating and proving the main result of this section, 
we need two more lemmas.
According to the guiding philosophy of \cite{MPPS}, the essentially rank $1$ 
part of the higher-rank resonance varieties of a $\dga$ is determined by rank $1$ 
resonance. We begin with a version of this general principle, valid for families 
of $\dga$ morphisms.

Fix an integer $q\ge 1$.  Let $\{\phi_f\colon A_f\to A\}_{f\in \cE}$ be 
a finite family of $(q-1)$-connected maps between connected 
$\C$-$\dga$s. Also, let $\g$ be a Lie algebra, and let 
$\theta\colon \g\to \gl(V)$ be a finite-dimensional representation. 
For each $i\le q$ such that $H^i(A)\ne 0$, Corollary 3.8 and Lemma 2.6 
from \cite{MPPS} give an inclusion 
\begin{equation}
\label{eq:riincl}
\RR^i_1(A,\theta) \supseteq  \Pi(A,\theta)\cup 
\bigcup_{f\in \cE}  (\phi_f\otimes \id) \RR^i_1(A_f,\theta) .
\end{equation}
\begin{lemma}
\label{lem:mppsfam}
Assume \eqref{eq:riincl} holds as an equality in the rank $1$ case.  Then 
\[
\RR^i_1(A,\theta) \cap \F^1(A,\g) \subseteq  \Pi(A,\theta)\cup 
\bigcup_{f\in \cE}  (\phi_f\otimes \id) \RR^i_1(A_f,\theta).
\]
\end{lemma}

\begin{proof}
Let $\omega=\eta\otimes g$ be a non-zero element in 
$\left( \RR^i_1(A,\theta) \cap \F^1(A,\g) \right) \setminus \Pi(A,\theta)$. 
From \cite[Cor.~3.8]{MPPS}, we know that $\eta\otimes g$ belongs to 
$\RR^i_1(A,\theta) \cap \F^1(A,\g)$ if and only if there is an eigenvalue 
$\lambda$ of $\theta(g)$ such that $\lambda \eta \in \RR^i_1(A)$. 
By our assumption on the rank $1$ resonance, $\lambda \eta=\phi_f(\eta')$, 
for some $\eta'\in \RR^i_1(A_f)$.  Since $\lambda\ne 0$ 
we infer that $\omega=(\phi_f \otimes \id) (\eta_f \otimes g)$, for some 
$\eta_f \in A^1_f$ such that $d \eta_f=0$ and $\lambda \eta_f \in \RR^i_1(A_f)$. 
Again by \cite[Cor.~3.8]{MPPS}, we conclude that 
$\eta_f \otimes g$ belongs to $\RR^i_1(A_f,\theta) \cap \F^1(A_f,\g)$, 
and we are done. 
\end{proof}

\begin{lemma}
\label{lem:flatres}
Let $A$ be a $1$-finite $\C$-$\dga$ with $d=0$, and let $\theta\colon \g\to \gl(V)$ 
be a finite-dimensional representation of a non-abelian Lie subalgebra of $\sl_2(\C)$. 
Then $\F(A,\g)=\F^1(A,\g)\cup \RR^1_1(A,\theta)$.
\end{lemma}

\begin{proof}
Let $\omega\in \F(A,\g)\setminus \F^1(A,\g)$.   It follows from the proof 
of \cite[Prop.~5.3]{MPPS} that  $\omega \in U \otimes \g$, 
where $U\subseteq A^1$ is a linear subspace of dimension 
at least $2$ which is isotropic with respect to the multiplication map 
$A^1\wedge A^1\to A^2$.   Clearly, $A_U:=\C\cdot 1\oplus U$ is a
finite sub-$\dga$ of $A$, and $\chi(H^{\hdot}(A_U))<0$.  
By \cite[Prop.~2.4]{MPPS}, we have that $\F(A_U,\g)=\RR^1_1(A_U,\theta)$.  
Therefore, $\omega \in \RR^1_1(A_U,\theta)\subseteq \RR^1_1(A,\theta)$, 
and this completes the proof.
\end{proof}

\begin{theorem}
\label{thm:kpencil}
Let $f\colon M\to N$ be the projection map of a principal $K$-bundle, 
where both $M$ and $N$ are smooth, closed manifolds, and $K$ 
is a compact, connected real Lie group. Let $G$ be a complex linear algebraic 
group, with non-abelian Lie algebra $\g\subseteq \sl_2(\C)$. 
Let $\iota \colon G \to \GL (V)$ be a rational representation. Let 
$f_{\sharp}\colon \pi_1(M)\surj \pi_1(N)$ be the induced homomorphism 
on fundamental groups, and let $\abf\colon \pi_1(M)\surj \pi_1(M)_{\abf}$ 
be the canonical projection. Suppose $N$ is formal. Then, 
\begin{align}
\label{eq:repincl-b}
\Hom(\pi_1(M),G)_{(1)} &= \abf^{!} \Hom(\pi_1(M)_{\abf},G)_{(1)} \cup 
f_{\sharp}^{!} \Hom(\pi_1(N),G)_{(1)}, 
\\
\intertext{and, for $i=r=1$ or $i=0$ and $r\ge 1$,}
\label{eq:vincl-b}
\VV^i_r(\pi_1(M),\iota)_{(1)}&= \abf^{!} \VV^i_r(\pi_1(M)_{\abf}, \iota)_{(1)} 
\cup f_{\sharp}^{!} \VV^i_r (\pi_1(N),\iota)_{(1)}.
\end{align}
\end{theorem}

\begin{proof}
By Corollary \ref{cor:jumpnat-top}, equality of germs at $1$ in \eqref{eq:repincl} 
implies equality at $1$ in \eqref{eq:vincl} for $i=0$ and $r\ge 1$.  Thus, 
in order to verify equality \eqref{eq:vincl-b}, it is enough to assume $i=r=1$.  
As we saw in the proof of Theorem \ref{thm:rk2k}, both our claims 
hold trivially when $b_1(M)=0$. 
Consequently, we may also assume that $b_1(M)>0$.  

Since the orbit space $N$ is formal, we may take as a model for it 
the $\dga$ $A_N=(H^{\hdot}(N), d=0)$.  
As usual, let $[\tau]\colon P^{\hdot}\to H^{\hdot+1}(N)$ 
be the transgression in the Serre spectral sequence of the fibration $K\to M\to N$. 
By Proposition \ref{prop:hmodel}, the Hirsch extension 
$A_M:=A_N\otimes_{\tau} \bigwedge P$ is a finite $\dga$ model for $M$,   
and  the canonical inclusion 
$\Phi\colon H^{\hdot}(N) \inj A^{\hdot}_M$ is a model for the  
map $\Omega(f)\colon \Omega(N) \to \Omega(M)$.  

Now set $A^{\hdot}_0=(\bwedgedot H^1(M), d=0)$, and let   
$f_0\colon M\to K(\pi_1(M)_{\abf},1)$ be the canonical map 
defined by the homomorphism $\abf$. 
By Theorem \ref{prop:2unif} (with $q=1$), 
there is a $\dga$ map $\Phi_0\colon A_0\to A_M$  
such that $\Omega(f_0)\simeq_1 \Phi_0$ and $\Omega(f)\simeq_1 \Phi$
in $\ACDGA_0$, uniformly with respect to the 
families $\{f, f_0\}$ and $\{\Phi, \Phi_0\}$. 
Since, as was mentioned in Proposition \ref{prop:hmodel}, 
both $f$ and $\Phi$ are $0$-connected, 
Theorem \ref{thm:main} applies, giving  
an analytic isomorphism of embedded germs, 
\begin{multline}
\label{eq:kpen1}
\left(\Hom(\pi_1(M),G) , \abf^{!} \Hom(\pi_1(M)_{\abf},G) \cup 
f_{\sharp}^{!} \Hom(\pi_1(N),G) \right)_{(1)} \cong 
\\  
\left(\F(A_M,\g) , (\Phi_0\otimes \id) \F(A_0,\g)\cup 
(\Phi\otimes \id) \F(A_N,\g) \right)_{(0)}.
\end{multline}

On the other hand, Proposition 5.3 from \cite{PS-15} guarantees the 
global equality
\begin{equation}
\label{eq:kpen2}
\F(A_M,\g) = \F^1(A_M,\g)\cup (\Phi\otimes \id) \F(A_N,\g).
\end{equation}

We may also apply Lemma \ref{lem:l4k} to the map $\Phi_0$  to deduce 
the global equalities 
\begin{align}
\label{eq:kpen3}
 \F^1(A_M,\g) &= (\Phi_0\otimes \id) \F(A_0,\g), \\
  \Pi(A_M,\theta) &= (\Phi_0\otimes \id) \RR^1_1(A_0,\theta). \notag
\end{align}

Using equations \eqref{eq:kpen1}--\eqref{eq:kpen3} as well as 
Theorem \ref{thm:main}, we see that, in order to complete the proof, 
it is enough to show the following: if the inclusion  
\begin{align}
\label{eq:kpen4}
\F(A_M,\g) &\subseteq  \F^1(A_M,\g)\cup (\Phi\otimes \id) \F(A_N,\g) 
\intertext{holds, then the inclusion}
\label{eq:kpen5}
\RR^1_1(A_M,\theta) &\subseteq  \Pi(A_M,\theta)\cup (\Phi\otimes \id) \RR^1_1(A_N,\theta)
\end{align}
also holds. Pick  $\omega \in \RR^1_1(A_M,\theta)$.  There are 
two cases to consider. 

First suppose that $\omega \in \F^1(A_M,\g)$. Using \cite[Prop.~5.5]{PS-15} 
and induction on the dimension of $P^1$, we see that $\RR^1_1(A_M)\subseteq 
\{0\}\cup \Phi (\RR^1_1(A_N))$.  Hence, we may apply 
Lemma \ref{lem:mppsfam} (for $i=q=1$) to the one-element family 
$\{\Phi\colon A_N\inj A_M\}$   and conclude that 
$\omega \in \Pi(A_M,\theta)\cup (\Phi\otimes \id) \RR^1_1(A_N,\theta)$, 
as required. 

Finally, suppose that $\omega \not\in \F^1(A_M,\g)$.   Then 
$\omega =\Phi\otimes \id (\omega')$, for some $\omega' \in \F(A_N,\g)  
\setminus$  $ \F^1(A_N,\g)$, by assumption \eqref{eq:kpen4}.  
By Lemma \ref{lem:flatres}, we have that $\omega'\in \RR^1_1(A_N, \theta)$. 
This verifies that \eqref{eq:kpen5} holds in this case, too, thereby 
completing the proof.
\end{proof}

Theorem \ref{thm:kpencil} improves on Theorem 1.5(2) from \cite{PS-15}, where 
an extra assumption  (injectivity of the transgression in degree $1$) was 
required. Our stronger result here is of the same flavor as  the equality 
\eqref{eq:repincl-k} from Theorem \ref{thm:rk2k}, in the context provided by  
Remark \ref{rem:tri}.  Namely, if $G$ 
is a $\C$-linear algebraic group with Lie algebra $\g$ as above, 
and if $\rho \colon \pi_1(M)\to G$ is a representation near the origin $1$, then 
$\rho$ is either abelian or pulls back  via $f$ from a $G$-representation 
of $\pi_1(N)$.

\section{Quasi-projective manifolds}
\label{sect:qproj}

We conclude with another interesting class of examples  
where the uniform property holds for one-element families 
of maps, namely, regular maps between  
smooth, quasi-projective varieties.  We also derive a non-compact 
analogue of Theorem \ref{thm:rk2k} for a special class of quasi-projective 
manifolds, namely, complements of complex hyperplane arrangements. 

\subsection{Mixed Hodge diagrams}
\label{subsec:mhd}

Let $M$ be an irreducible, smooth, complex quasi-projec\-tive variety, 
or, for short, a {\em quasi-projective manifold}. Note that $M$ is a finite space.
By resolution of singularities,
we have that $M=\overline{M}\setminus D$, where $\overline{M}$ is a 
smooth projective variety, and $D\subset \overline{M}$ is a normal crossing 
divisor. A map between such pairs, $\bar{f}\colon (\overline{M},D) \to (\overline{M}',D')$, is 
called a {\em regular morphism}\/ if the map $\bar{f}\colon \overline{M}\to \overline{M}'$ 
is a regular map with the property that $\bar{f}^{-1}(D')\subseteq D$. Clearly, the restriction 
$f\colon \overline{M}\setminus D \to \overline{M}'\setminus D'$ is also a regular map. 
Conversely, any regular map between quasi-projective manifolds is induced by a
regular morphism between convenient compactifications with normal crossing 
divisors. 

We want to prove a quasi-projective analogue of Proposition \ref{prop:mainf}. 
For that, we will need the theory of relative minimal models for mixed Hodge 
diagrams (MHDs, for short), developed by Cirici and Guill\'{e}n in \cite{CG}. 
We start by recalling some pertinent definitions and results from \cite{CG}. 

The objects of the category $\FDGA$ are of the form 
$(A^{\hdot}, W_{\sdot})$, where $(A^{\hdot}, d)$ is a $\dga$ defined over $\Q$ 
and $W_{\sdot}$ is an increasing, multiplicative, regular, exhaustive filtration 
on $(A^{\hdot}, d)$, called a {\em weight}\/ filtration.  
Such an object gives rise to a spectral sequence  in the category of 
bigraded $\dga$s, $\{E_r(A)\}_{r\ge 1}$, which converges to $H^{\hdot}(A)$.  
A morphism in $\FDGA$ is a $\dga$ map which respects filtrations.  Such a 
morphism $\psi$ induces a map of spectral sequences, $\{E_r(\psi)\}_{r\ge 1}$.

The objects of the category $\MHD$ are strings of morphisms 
in $\FDGA$ defined over $\C$, 
\begin{equation}
\label{eq:filtelem}
\xymatrixcolsep{24pt}
\HH:\ \xymatrix{
A_0  \ar^{\psi_0} [r] & A_1  & \ar[l] \cdots \ar[r] 
& A_{\ell-1}  & A_{\ell} \ar_(.4){\psi_{\ell-1}}[l] 
},
\end{equation}
where $(A_0,W)$ is defined over $\Q$ and all the induced maps 
$E_2(\psi_i)$ are isomorphisms. There are also additional data and 
axioms, related to the mixed Hodge structure (MHS) on 
$A_{\ell}$, see \cite[Def.~3.1]{CG}. A morphism of mixed Hodge diagrams, 
$\Phi\colon \HH'\to \HH$, is a tuple of $\fdga$ maps, $(\Phi_0,\Phi_1,\dots,\Phi_{\ell})$, 
commuting with the maps $\psi'$ and $\psi$, and such that $\Phi_0$ 
is defined over $\Q$.  There is also an extra condition on $\Phi_{\ell}$ 
pertaining to the MHS, see \cite[Def.~3.5]{CG}.

\subsection{The Gysin model of Morgan and Navarro}
\label{subsec:gysin}

Returning to our setup, let $M$ be a quasi-projective manifold, 
and let $\overline{M}=M\cup D$ be a normal-crossing compactification. 
Given these data, Navarro constructs in \cite{N} a mixed Hodge diagram 
$\HH(\overline{M},D)$, functorial with respect to regular morphisms of 
pairs (see also Hain \cite{H1}). Furthermore, there is an equivalence 
$A_0(\overline{M},D)\simeq \Omega_{\C}(M)$ in $\CDGA$, natural 
with respect to the pair $(\overline{M},D)$.  Moreover, 
$E_1(A_0(\overline{M},D))$ is isomorphic (as a bigraded $\dga$) 
to $\MG(\overline{M},D)$, the {\em Gysin model}\/ of 
$M=\overline{M}\setminus D$ constructed by Morgan 
in \cite{Mo, Mo-86} (see also Dupont \cite{Du}).
Note that this is a finite $\C$-model, defined 
over $\Q$, and that 
$\MG(\overline{M},\emptyset)=(H^{\hdot}(\overline{M}),d=0)$.  

Suppose $\bar{f}\colon (\overline{M},D) \to (\overline{M'},D')$ is a regular 
morphism, such that the restriction $f\colon M\to M'$ 
preserves basepoints. Naturality in the sense of Navarro yields an equivalence 
\begin{equation}
\label{eq:phizero}
\Omega_{\C} (f) \simeq \Phi_0(\bar{f}) 
\end{equation}
in $\C$-$\ACDGA_0$. Following Cirici and Guill\'{e}n \cite{CG}, 
we define  
\begin{equation}
\label{eq:mgnav}
\Phi (\overline{f})= E_1(\Phi_0 (\overline{f})) \colon 
\MG(\overline{M} ',D') \to \MG(\overline{M},D),
\end{equation} 
over $\C$. 

\begin{prop}
\label{lem:mainqp}
Let $f\colon M\to M'$ be a pointed, regular map between quasi-projective 
manifolds, inducing an injection on $H^1$.   Extend $f$ to a regular  
morphism, $\bar{f}\colon (\overline{M},D) \to (\overline{M'},D')$, 
by adding divisors with normal crossings in a suitable manner. 
Then $\Omega_{\C}(f) \simeq \Phi (\overline{f})$ in $\C$-$\ACDGA_0$. 
\end{prop}

\begin{proof}
Looking at $\Q$-components of MHDs and ignoring additional 
MHS data, we extract from \cite[Theorems 3.17 \& 3.19]{CG} the 
following commuting square in $\FDGA$:
\begin{equation}
\label{eq:hrelmodel}
\begin{gathered}
\xymatrixcolsep{30pt}
\xymatrix{
A_0(\overline{M},D) &\bigwedge U' \otimes \bigwedge U 
\ar_(.52){\rho}[l] 
\\
A_0(\overline{M'},D') \ar^{\Phi_0(\bar{f})}[u] &\bigwedge U'\ar_(.45){\rho'}[l] 
\ar@{^{(}->}^{j}[u] 
}
\end{gathered}
\end{equation}

By \cite[Lemma~3.4]{CG}, the induced maps $E_2(\rho)$ 
and $E_2(\rho')$ are known to be isomorphisms. Hence, the maps 
$\rho$ and $\rho'$ are quasi-isomorphisms. Furthermore, the 
$\CDGA$ diagram underlying \eqref{eq:hrelmodel} has the 
following properties: $\rho'$ is a minimal model map, and $\rho$ 
is a relative minimal model map for $\Phi_0(\bar{f}) \circ \rho'$, 
in the sense of \S\ref{subsec:minmod}. 

Our injectivity assumption on $H^1(f)$, together with the equivalence 
from \eqref{eq:phizero}, 
imply that the map $\Phi_0(\bar{f}) \circ \rho'$ is a $0$-equivalence.  
Using the discussion from  \S\ref{subsec:minmod}, we infer that 
both $\bigwedge U'$ and $\bigwedge U' \otimes \bigwedge U$ 
are connected $\dga$s. In particular, all maps from diagram \eqref{eq:hrelmodel} 
respect augmentations. 

It's time now to take into account the available MHS data.  We know from 
the work of Cirici and Guill\'{e}n that the map $j$ is actually a morphism of 
mixed Hodge $\dga$s, in the sense of \cite[Definition 3.14]{CG}.  
According to Deligne's functorial splitting over $\C$ of mixed Hodge 
structures, we have the following identifications in $\CDGA$, 
\begin{equation}
\label{eq:deligne}
E_1(\bwedge U') = \bwedge U', \quad 
E_1(\bwedge U'\otimes \bwedge U) = \bwedge U' \otimes \bwedge U, \quad
E_1(j)=j.
\end{equation}
This can be verified using the argument of Morgan from \cite[Thm.~9.6]{Mo}.   
See also \cite[Lemma 3.20]{CG}, where no extra finite-type assumptions 
are needed (over $\C$).

Applying the $E_1$ functor to diagram \eqref{eq:hrelmodel}, we obtain 
the following commuting diagram in $\CDGA$,
\begin{equation}
\label{eq:e1sq}
\begin{gathered}
\xymatrixcolsep{30pt}
\xymatrix{
E_1(\bigwedge U' \otimes \bigwedge U)
\ar^(.60){E_1(\rho)}[r] 
&E_1(A_0)
\ar@{=}[r]
& \MG(\overline{M},D) 
\\
E_1(\bigwedge U') 
\ar^{E_1(j)}[u] \ar^(.60){E_1(\rho')}[r] 
&E_1(A_0')  
\ar_{E_1(\Phi_0)}[u]  \ar@{=}[r]
& \MG(\overline{M'},D')
\ar_{\Phi (\bar{f})}[u] 
}
\end{gathered}
\end{equation}

Here both horizontal arrows are quasi-isomorphisms, since 
$E_2(\rho)$ and $E_2(\rho')$ are isomorphisms.  Since all 
$\dga$s in sight are connected, \eqref{eq:e1sq} is a commuting 
diagram in $\ACDGA$. The desired conclusion follows by putting 
together the information from displays \eqref{eq:phizero} and
\eqref{eq:hrelmodel}--\eqref{eq:e1sq}. 
\end{proof}

\begin{remark}
\label{rem:deligne}
As mentioned previously, it is known that the Navarro model 
$E_1(A_0(\overline{M},D))$ is isomorphic in $\CDGA$ 
to Morgan's Gysin model $\MG(\overline{M},D)$. It is also known that 
the latter is functorial with respect to regular morphisms of pairs; see
\cite{Du} for a convenient, explicit description of the $\dga$ map 
$\MG (\bar{f}) \colon \MG(\overline{M} ',D') \to \MG(\overline{M},D)$
induced by $\bar{f}\colon (\overline{M},D) \to (\overline{M'},D')$.
But we do not know whether under this identification on objects the map 
$\MG(\bar{f})$ coincides with the map $\Phi (\bar{f})$ defined in \eqref{eq:mgnav}. 
If that were the case, one could use \cite[Ex.~5.3]{DP-ccm} to infer 
that the map $\Phi (\bar{f})=\MG(\bar{f})$ is injective, whenever $f\colon M\to M'$
is a regular surjection onto a curve, with connected generic fiber.
This observation, together with Proposition \ref{lem:mainqp}, would 
then imply that the conclusions of Corollary \ref{cor:unione} 
hold for regular admissible maps defined on quasi-projective manifolds, 
in the case when $q=1$. 
\end{remark}

\subsection{Hyperplane arrangements}
\label{subsec:arrs}

Let $\AA$ be an arrangement of hyperplanes, that is, a finite, 
non-empty collection of complex affine hyperplanes in $\C^{\ell}$, for some $\ell>0$.   
The union of these hyperplanes is an affine 
hypersurface, $V_\AA$,  defined by an equation of the form $Q_{\AA}=0$, 
where $Q_{\AA}=\prod_{H\in \AA} \alpha_H$ and $\alpha_H=0$ is a linear 
equation defining the hyperplane $H$.   
The complement of the arrangement, $M_\AA=\C^{\ell}\setminus V_\AA$, 
is a connected, smooth, quasi-projective variety, which has the homotopy type 
of a finite CW-complex of dimension at most $\ell$. 

A nice feature of this class of quasi-projective manifolds is that 
formality over $\k=\R$ or $\C$ holds in the following strong sense. 
For each $H\in \AA$, the logarithmic $1$-form
\begin{equation}
\label{eq:eh}
\xi_H = \frac{1}{2\pi \ii}\, d \log \alpha_H \in \Omega_{\DR}(M_{\AA})
\end{equation}
is a closed form.  Let $e_H\in H^1(M_{\AA},\k)$ be the cohomology class 
corresponding to $[\xi_H]\in H^1_{\DR}(M_{\AA})$ under the 
de Rham isomorphism.  It is known that $\{e_H \mid H\in \AA\}$ 
forms a basis for $H^1(M_{\AA},\k)$.  Thus, the $\k$-linear map  
$\xi_{\AA}\colon H^1(M_{\AA},\k) \to \Omega^1_{\DR}(M_{\AA})$ 
sending each $e_H$ to $\xi_H$ yields an isomorphism 
$[\xi_{\AA}] \colon H^1(M_{\AA},\k) \isom H^1_{\DR}(M_{\AA})$.

The celebrated Brieskorn--Orlik--Solomon theorem (see \cite{OT}) 
states that the cohomology ring $H^{\hdot}(M_{\AA},\Z)$ 
is the quotient of the exterior algebra $\bwedgedot H^{1}(M_{\AA},\Z)$ 
by an ideal generated in degrees at least $2$ and depending only on 
the intersection lattice of $\AA$. Moreover,
the extension of $\xi_{\AA}$ to a $\dga$ map, 
$\xi_{\AA} \colon (\bwedgedot H^{1}(M_{\AA},\k),d=0) \to 
\Omega^{\hdot}_{\DR}(M_{\AA})$,  factors through a  quasi-isomorphism 
\begin{equation}
\label{eq:os}
\xymatrix{\xi_{\AA}\colon  (H^{\hdot}(M_{\AA},\k) , d=0) \ar[r]& 
\Omega^{\hdot}_{\DR}(M_{\AA})}.
\end{equation}

We now suppose that the arrangement $\AA$ is {\em central}, i.e., all
hyperplanes $H\in \AA$ pass through the origin $0\in \C^{\ell}$.
For the purpose of studying the fundamental group $\pi=\pi_1(M_{\AA})$, 
we may assume that $\AA$ is a central arrangement in $\C^3$. This 
can be achieved by taking a generic $3$-slice  (if $\ell>3$), or taking the 
product with $\C^{3-\ell}$ (if $\ell<3$); neither operation changes the 
fundamental group of the complement.

Recall that $M_{\AA}$ is a quasi-projective manifold.  
By the discussion from \S\ref{subsec:arapura}, there is a 
finite set $\cE(M_{\AA})$ of admissible maps $f\colon M_{\AA} \to M_f$ 
(up to reparametrization at the target), such that $M_f$ is a smooth 
curve with $\chi(M_f)<0$.  It turns out that the mixed Hodge structure 
on $M_{\AA}$ is pure of weight $2$.  Consequently, each curve $M_f$ must be of the 
form $\CP^1 \setminus \{\text{$k$ points}\}$, for some $k\ge 3$.  

Falk and Yuzvinsky gave in \cite{FY} a particularly nice, combinatorial 
description of the set $\cE(M_{\AA})$, well-suited for our purposes here 
(see also \cite[\S 5]{DS} and \cite[\S 6]{PS-14}). The key combinatorial 
notion is that of a multinet. Given an integer $k\ge 3$, a {\em $k$-multinet}\/ 
$\NN$ on a central arrangement $\AA$ in $\C^3$ consists of a 
partition, $\AA=\AA_1\sqcup \cdots \sqcup \AA_k$, and a multiplicity function, 
$m\colon \AA \to \N$, satisfying several axioms, one of which being 
that the sum $\sum_{H\in \AA_{i}} m_H$ is independent of $i\in [k]$. 

The multinet axioms imply that the polynomials $Q_i=\prod_{H\in \AA_i} \alpha_H^{m_H}$ 
belong to a pencil of curves, that is, for each $i>2$ there are constants 
$a_i$ and $b_i$ such that $Q_i=a_i Q_1+b_i Q_2$. 
Consider the central line arrangement $\LL=\{\LL_1,\dots ,\LL_k\}$ in $\C^2$, 
with $\LL_i=\{g_i=0\}$, where $g_1=z_1$, $g_2=z_2$,  
and $g_i=a_i z_1+b_i z_2$ for $i>2$.  Let $f_{\NN}\colon M_{\AA}\to M_{\LL}$ 
be the regular map  with components $(Q_1,Q_2)$.  Projectivizing, 
we obtain an admissible map, 
\begin{equation}
\label{eq:mpen}
f_{\NN}\colon M_{\AA}\to \CP^1 \setminus \{\text{$k$ points}\}.
\end{equation}

More generally, there is a complete set of representatives for $\cE(M_{\AA})$ consisting of
admissible maps $f_{\NN} \colon M_{\AA}\to \CP^1 \setminus \{\text{$k$ points}\}$
obtained by restricting to $M_{\AA}$ the map $f_{\NN}\colon M_{\B}\to \CP^1 \setminus \{\text{$k$ points}\}$,
where $\NN$ is a $k$-multinet on a sub-arrangement $\B\subseteq \AA$; see \cite[Corollary 6.6]{PS-14}.

Now set $n=\abs{\AA}$, and identify $\pi_{\abf}=\Z^n$.  
Let $\CC$ be the Boolean arrangement in $\C^n$, consisting of all 
coordinate hyperplanes.  Clearly, $M_{\CC}=(\C^{\times})^n$.   Consider 
the regular map $f_0\colon M_{\AA} \to M_{\CC}$  with components 
$(\alpha_H)_{H\in \AA}$. As noted for instance in \cite[Lem.~5.1]{DS}, 
the induced homomorphism, $(f_0)_{\sharp}\colon 
\pi_1(M_{\AA})\to \pi_1(M_{\CC})$, coincides with the canonical projection, 
$\abf \colon \pi\surj \pi_{\abf}$.  
We let $E(M_{\AA})=\cE(M_{\AA})\cup \{f_0\}$, as in \eqref{eq:em}.

By construction, all maps $f\in E(M_{\AA})$ are of the form  $f\colon M\to M_f$, where 
$M=M_\AA$ and $M_f=M_{\AA_f}$, for some (affine) arrangement $\AA_f$.  

\begin{prop}
\label{prop:unifarr}
Let $\AA$ be a central hyperplane arrangement in $\C^3$, and fix a basepoint 
in $M_\AA$.  For $\k=\R$ or $\C$,  we have that $\Omega_{\k}(f)\simeq H^{\hdot}(f,\k)$ 
in $\k$-$\ACDGA_0$, uniformly with respect to $f\in E(M_{\AA})$. 
\end{prop}

\begin{proof}
In view of the Brieskorn--Orlik--Solomon isomorphism \eqref{eq:os}, 
it is enough to check the commutativity of the following diagram 
in $\CDGA$:
\begin{equation}
\label{eq:natos}
\xymatrixcolsep{30pt}
\xymatrix{
(H^{\hdot}(M_{\AA}) , d=0) \ar^(.58){\xi_{\AA}}[r]& 
\Omega^{\hdot}_{\DR}(M_{\AA})\\
(H^{\hdot}(M_{\AA_f}) , d=0) \ar^(.58){\xi_{\AA_f}}[r] \ar^{H^{\hdot}(f)}[u] 
& \Omega^{\hdot}_{\DR}(M_{\AA_f}) \ar_{\Omega^{\hdot}_{\DR}(f)}[u]
}
\end{equation}

Since the cohomology ring of an arrangement complement is generated 
in degree $1$, we may assume that $\hdot=1$ in the above diagram. 
Using the explicit construction of the map $\xi$ in degree $1$, we can 
further reduce to showing that $\Omega_{\DR}(f) (d \log \alpha_{H'})$ belongs to the 
$\Z$-span of $\{ d \log \alpha_H  \mid H\in \AA\}$, for every $H'\in \AA_f$. 

First assume $f=f_0$. Then the claim follows  from 
the formula $\Omega_{\DR}(f_0) (d \log z_H) = d\log \alpha_H$, for 
every $H\in \AA$, which in turn follows directly from the 
definition of $f_0$. 

Next assume $f=f_{\NN}$, for some multinet $\NN$ on a sub-arrangement 
$\B\subseteq \AA$.  Clearly, we may assume that $\B=\AA$.  The claim 
is now an easy consequence of the formula 
$\Omega_{\DR}(f_{\NN}) (d \log g_i) = \sum_{H\in \AA_i} m_H \, d\log \alpha_H$, 
which is verified in \cite[Lem.~6.3]{PS-14}. 
\end{proof}

\begin{theorem}
\label{thm:rk2arr}
Let $\AA$ be a central hyperplane arrangement with complement $M=M_{\AA}$. 
Write $\pi=\pi_1(M)$, and, for each map $f\colon M\to M_f$ in $E(M)$, 
set $\pi_f=\pi_1(M_f)$.   Let $G$ be a $\C$-linear algebraic 
group with non-abelian Lie algebra $\g\subseteq \sl_2(\C)$, 
and let $\iota\colon G\to \GL(V)$ be a rational representation.  Then, 
\begin{align}
\label{eq:repincl-arr}
\Hom(\pi,G)_{(1)} &= \bigcup_{f\in E(M)} f_{\sharp}^{!} \Hom (\pi_f,G)_{(1)}, 
\\
\intertext{and, for $i=r=1$ or $i=0$ and $r\ge 1$,}
\label{eq:vincl-arr}
\VV^i_r(\pi,\iota)_{(1)}&= \bigcup_{f\in E(M)} f_{\sharp}^{!} \VV^i_r (\pi_f,\iota)_{(1)}.
\end{align}
\end{theorem}

\begin{proof}
As noted before, we may assume $\ell=3$.
The argument we give is closely modeled on the proof of Theorem \ref{thm:rk2k}. 
To begin with, note that the conclusions of Lemma \ref{lem:l3k} hold for the 
formal, quasi-projective manifold $M=M_\AA$, with the same proof.  Next, 
consider the map $f_0\colon M_\AA \to M_\CC$, and the induced 
$\dga$ map $H^{\hdot}(f_0)\colon H^{\hdot}(M_\CC) \to 
H^{\hdot}(M_\AA) $, where both differentials are $0$.  
Since $M_\CC=(\C^{\times})^n$, where $n=\abs{\AA}$, 
we may identify  $H^{\hdot}(M_\CC)$ with  $\bwedgedot H^1(M_\CC)$. 
Furthermore, $H^1(f_0)$ is an isomorphism, by construction. 
Hence, Lemma \ref{lem:l4k} may be applied to the map $\Phi=H^{\hdot}(f_0)$. 

By Proposition \ref{prop:unifarr}, we have that $\Omega(f)\simeq H^{\hdot}(f)$ 
in $\ACDGA_0$, uniformly with respect to $f\in E(M_{\AA})$. 
We may now apply Theorem \ref{thm:main} for  $q=1$ to the family of 
pointed continuous maps $\{f\colon M_{\AA} \to M_{\AA_f}\}_{f\in E(M_\AA)}$ 
and the family of $\dga$ maps $\{H^{\hdot}(f)\colon H^{\hdot}(M_{\AA_f}) \to 
H^{\hdot}(M_\AA)\}_{f\in E(M_\AA)}$, where again all differentials are $0$. 
The rest of the argument goes exactly as in the proof of Theorem \ref{thm:rk2k}.
\end{proof}

\section*{Acknowledgement}
Part of this work was done while the second author visited the 
Institute of Mathematics of the Romanian Academy in June, 2016.  
He thanks IMAR for its hospitality, support, and excellent 
research atmosphere. 


\newcommand{\arxiv}[1]
{\texttt{\href{http://arxiv.org/abs/#1}{arxiv:#1}}}
\newcommand{\arx}[1]
{\texttt{\href{http://arxiv.org/abs/#1}{arXiv:}}
\texttt{\href{http://arxiv.org/abs/#1}{#1}}}
\newcommand{\doi}[1]
{\texttt{\href{http://dx.doi.org/#1}{doi:#1}}}
\renewcommand{\MR}[1]
{\href{http://www.ams.org/mathscinet-getitem?mr=#1}{MR#1}}

\end{document}